\documentclass[1996/10/24]{amsart}

\usepackage{color}
\usepackage{amssymb,stmaryrd}
\usepackage{graphicx}

\usepackage{color}
\definecolor{black}{rgb}{0,0,0}

\definecolor{red}{rgb}{1,0,0}

\definecolor{blue}{rgb}{0,0,1}

\newcommand{\nn}{\nonumber}
\def\Real{\mathbb{R}}

\newcommand{\Ct}{{\mathcal T}}

\newcommand{\Lb}{\{\hspace{-4.0pt}\{}
\newcommand{\Rb}{\}\hspace{-4.0pt}\}}

\def\bfz{\textit{\textbf{z}}}

\def\bfn{\textit{\textbf{n}}}

\def\bfu{\textit{\textbf{u}}}

\def\bfv{\textit{\textbf{v}}}
\def\bfw{\textit{\textbf{w}}}
\def\bfx{\textit{\textbf{x}}}

\def\bfe{\textit{\textbf{e}}}

\def\bfH{\textit{\textbf{H}}}

\def\bfP{\textit{\textbf{P}}}

\def\bfW{\textit{\textbf{W}}}

\def\bfPi{\boldsymbol{\Pi}}

\def\bfeta{\boldsymbol{\eta}}

\def\bfphi{\boldsymbol{\phi}}
\def\bfpsi{\boldsymbol{\psi}}
\def\bfpsi{\boldsymbol{\psi}}

\def\calF{\mathcal{F}}

\def\calT{\mathcal{T}}

\def\calE{\mathcal{E}}
\def\calK{\mathcal{K}}

\usepackage{url}

\newtheorem{thm}{Theorem}[section]

\newtheorem{lem}[thm]{Lemma}
\theoremstyle{remark}
\newtheorem{rem}[thm]{Remark}

\newcommand{\bu}{{\bf u}}

\newtheorem{example}[thm]{Example}
\newtheorem{remark}{Remark}[section]
\numberwithin{equation}{section}
\numberwithin{figure}{section}
\begin{document}

\title[MFEM for biharmonic  and von K\'arm\'an equations]{A mixed finite element scheme for biharmonic 
equation with variable coefficient and von K\'arm\'an equations}

\author{Huangxin Chen}
\address{School of Mathematical Sciences and Fujian Provincial Key Laboratory on Mathematical Modeling and 
High Performance Scientific Computing, Xiamen University, Fujian, 361005, China}
\email{chx@xmu.edu.cn}

\author{Amiya K. Pani}
\address{Department of Mathematics, Indian Institute of Technology,
Bombay, Powai, Mumbai-400076, India} 
\email{akp@math.iitb.ac.in}

\author{Weifeng Qiu}
\address{Department of Mathematics, City University of Hong Kong,
83 Tat Chee Avenue, Kowloon, Hong Kong, China}
\email{weifeqiu@cityu.edu.hk}

\thanks{The work of Huangxin Chen was supported by the NSF of China (Grant No. 11771363) and the Fundamental Research Funds for the Central Universities (Grant No. 20720180003). The work of Amiya K. Pani is supported  by IITB Chair Professor's  fund and also partly by a MATRIX Grant No. MTR/201S/000309 (SERB, DST, Govt. India). Weifeng Qiu is supported by a grant from the Research Grants Council of the Hong Kong Special Administrative Region, China (Project No. CityU 11302219). The third author is the corresponding author.}

\subjclass[2010]{}


\date{}

\begin{abstract}
In this paper, a new mixed finite element scheme using element-wise stabilization is introduced 
for the biharmonic equation with variable coefficient on Lipschitz polyhedral domains. The proposed scheme 
doesn't involve any integration along mesh interfaces. The gradient of 
the solution is approximated by $H({\rm div})$-conforming $BDM_{k+1}$ element or vector valued Lagrange 
element with order $k+1$, while the solution is approximated by Lagrange element with order $k+2$ 
for any $k\geq 0$.This scheme can be easily implemented and produces 
positive definite linear system. We provide a new discrete $H^{2}$-norm stability, which is useful not only in 
analysis of this scheme but also in $C^{0}$ interior penalty methods and DG methods. 
Optimal convergences in both discrete $H^{2}$-norm and $L^{2}$-norm are derived. This scheme  with 
its analysis is further generalized to the von K\'arm\'an equations. Finally, numerical results verifying the theoretical 
estimates of the proposed algorithms are also presented.
\end{abstract}

\maketitle

\section{Introduction}
In the first part of this paper, a new mixed finite element scheme is proposed and analyzed 
for the following biharmonic equation with variable coefficient:
\begin{subequations}\label{equ_biharmonic}
\begin{align}
\label{equ_biharmonic1}
\Delta \left( \kappa \Delta u \right) = f, &  \ \ \textrm{in}  \ \ \Omega,  \\
\label{equ_biharmonic_dirichlet}
u = 0, &  \ \ \textrm{on}  \ \  \partial \Omega,\\
\label{equ_biharmonic2}
\frac{\partial u}{\partial n} = 0, &  \ \ \textrm{on}  \ \  \partial \Omega,
\end{align}
\end{subequations}
where $\Omega \subset \mathbb{R}^d (d\in \mathbb{N})$ is a Lipschitz polygonal or polyhedral domain, 
the coefficient $\kappa \in W^{1,\infty}(\Omega)$ such that  $0<\kappa_{0}  \leq \kappa(\boldsymbol{x}) 
\leq \kappa_{1}$,  and $f \in H^{-1}(\Omega)$. 
By using element-wise stabilization, our scheme doesn't involve any integration along mesh interfaces. 
Our scheme uses $H(\text{div})$-conforming $BDM_{k+1}$ or vector valued Lagrange element with 
order $k+1$ to approximate $\bfw = \nabla u$, and approximates $u$ with Lagrange element with 
order $k+2$ for any $k\geq 0$.
The second part of this paper is related to 
an application of our scheme to the von K\'arm\'an model, which can be stated  as follows:
\begin{subequations}\label{equ_von_karman}
\begin{align}
\Delta^2 \xi - [\xi,\psi]= f, &  \ \ \textrm{in}  \ \ \Omega, \label{vk_eq1}\\
\Delta^2 \psi  + [\xi,\xi]= 0, &  \ \ \textrm{in}  \ \ \Omega,\label{vk_eq2} \\
\xi =\frac{\partial \xi}{\partial n}  = 0, &  \ \ \textrm{on}  \ \  \partial \Omega, \label{vk_bc1}\\
\psi = \frac{\partial \psi}{\partial n} = 0, &  \ \ \textrm{on}  \ \  \partial \Omega,\label{vk_bc2}
\end{align}
\end{subequations}
where $\Omega \subset \mathbb{R}^2$ is a Lipschitz polygonal domain, $f \in  H^{-1}(\Omega)$, 
and the von K\'arm\'an bracket $[\cdot,\cdot]$ appearing in (\ref{vk_eq1}) and (\ref{vk_eq2}) is defined by
\[
[\eta,\phi] = \frac{\partial^2 \eta}{\partial x_1^2} \frac{\partial^2 \phi}{\partial x_2^2} + \frac{\partial^2 \eta}{\partial 
x_2^2} \frac{\partial^2 \phi}{\partial x_1^2} - 2 \frac{\partial^2 \eta}{\partial x_1 \partial x_2} 
\frac{\partial^2 \phi}{\partial x_1 \partial x_2} = {\rm cof}(D^2 \eta): D^2 \phi.
\]
Here ${\rm cof}(D^2 \eta)$ denotes the cofactor matrix of the Hessian of $\eta$ and $A : B$ denotes 
the Frobenius inner product of the matrices $A$ and $B$.


In literature, there are many numerical methods available for the biharmonic equation, that is, the problem  
(\ref{equ_biharmonic}) with $\kappa = 1$. Some of them can be easily generalized to include biharmonic problem 
with variable coefficients. We provide below a brief summary of results which are relevant to our present investigation.
\begin{itemize}

\item {\it Numerical methods approximating both $u$ and $\Delta u$}. The Ciarlet and Raviart (C-R) method 
\cite{CRmethod} uses $u$ and $\Delta u$ as unknowns and thereby, gives rise to  a system of Poisson problems. 
Then, $H^{1}$-conforming finite element spaces are used to approximate both $u$ and $\Delta u$, and  
it has no stabilization along mesh interfaces. Thus, the C-R method can be easily implemented. 
The stability of the C-R method with respect to discrete $H^{2}$-norm is shown in \cite{BabushkaCR}. However, 
the analysis in \cite{BabushkaCR} requires that the domain is convex (see, the proof of \cite[Lemma~$5$]{BabushkaCR}). 
The optimal convergence to $u$ of the C-R method, which  is obtained in \cite{Scholtz1978} though the convergence to 
$\Delta u$ is suboptimal. Though the corresponding linear system of the C-R method is a saddle point one, its 
conditional number may be of order $O(h^{-2})$ if numerical approximations to $u$ and $\Delta u$ are calculated 
alternatively. Optimal convergence to $\Delta u$ is obtained by the method \cite{Falk1978} 
which is similar to the C-R method. But the analysis in \cite{Falk1978} doesn't provide the stability 
with respect to discrete $H^{2}$-norm. For $hp$-mixed DG method  with penalization of the interelement boundary 
jump terms applied to the split system, see \cite{GNP2008}.By approximating $\kappa \Delta u$ instead of $\Delta u$, 
all these methods can be easily generalized for variable coefficient $\kappa$. 

\item {\it Numerical methods approximating both $u$ and $D^{2}u$}. The Hellan-Herrmann-Johnson (HHJ) method 
analyzed by Johnson in \cite{Johnson1973} treats $u$ and $D^{2}u$ as unknowns. It uses $H^{1}$-conforming 
approximations to $u$ and normal-normal continuous symmetric approximations to $D^{2}u$. Optimal convergence to 
both $u$ and $D^{2}u$ is shown in \cite{Johnson1973}. Since it naturally provides the stability with respect to discrete 
$H^{2}$-norm, the HHJ method and its variants are suitable for solving the von K\'arm\'an model 
(see \cite{Miyoshi1976, Brezzi1981, Reinhart1982}). We notice that the method  \cite{BehrensGuzman} 
uses the same formulation by the HHJ method, but writes the biharmonic equation as four first-order equations 
instead of two second-order equations. The method \cite{BehrensGuzman} obtains optimal convergence 
to $u$, $\nabla u$ and $D^{2}u$, and its global unknowns after hybridization are numerical approximations 
to the trace of $u$ and $\nabla u$ along the mesh interfaces. Overall, the analysis of the HHJ method 
and its variant (including the method \cite{BehrensGuzman}) can be generalized to the von K\'arm\'an equations 
(\ref{equ_von_karman}). But the implementation may not be easy and the corresponding linear system 
(without hybridization) is a saddle point system 
with a large number of degrees of freedom. Furthermore, since all these methods utilize the following identity
\begin{align}
\label{biharmonic_identity}
\int_{\Omega} \Big(\frac{\partial^{2}u}{\partial x_{1}^{2}} \frac{\partial^{2} v}{\partial x_{2}^{2}} 
+ \frac{\partial^{2}u}{\partial x_{2}^{2}} \frac{\partial^{2} v}{\partial x_{1}^{2}} 
- 2 \frac{\partial^{2}u}{\partial x_{1}\partial x_{2}} \frac{\partial^{2} v}{\partial 
x_{1}\partial x_{2}}\Big)\; d\bfx = 0 
\qquad \forall u, v \in H_{0}^{2}(\Omega),
\end{align} 
it may not be straightforward to generalize these methods for problems with a non-constant coefficient $\kappa$. 
One way is to split the biharmonic term $\Delta (\kappa \Delta u)$ as 
\begin{align}
\label{operator_split1}
\Delta (\kappa \Delta u) = \underline{\kappa} \Delta^{2} u +  \Delta ((\kappa - \underline{\kappa})\Delta u),  
\end{align}
where $\underline{\kappa}$ is a positive constant chosen to satisfy
\begin{align*}
\underline{\kappa} \leq \inf_{\bfx \in \Omega} \kappa (\bfx). 
\end{align*}
Applying (\ref{biharmonic_identity}) to the first term on the right hand side of (\ref{operator_split1}), it is easy to see that 
the solution $u$ of the biharmonic equation (\ref{equ_biharmonic}) satisfies  
\begin{align}
\label{biharmonic_identity_var}
\dfrac{\underline{\kappa}}{2} (D^{2}u, D^{2}v)_{\Omega} + ( (\kappa - \underline{\kappa})\Delta u, \Delta v)_{\Omega} 
= (f ,v)_{\Omega},\qquad \forall v \in H_{0}^{2}(\Omega).
\end{align}
Based on the variational formula (\ref{biharmonic_identity_var}), all HHJ type methods can be generalized for 
variable coefficient $\kappa$. However, the value of $\underline{\kappa}$ may affect the stability 
of HHJ type methods, if $\underline{\kappa}$ is chosen to be much smaller than 
$\inf_{\bfx \in \Omega} \kappa (\bfx)$ which may not be easy to discover in practice.

\item {\it Numerical methods approximating $u$ only}. There are several sub-classes of numerical methods 
approximating $u$ only. These methods can produce symmetric and positive definite linear system but 
with condition number of order $O(h^{-4})$ (the same as our mixed finite element scheme). 
One class uses $C^{1}$-conforming finite element spaces, which are naturally suitable for the biharmonic equation 
\cite{Argyris1968, Bogner1965, Douglas1979} and the von K\'arm\'an equations \cite{Brezzi1978, Miyoshi1976}. 
The main drawback of $C^{1}$-conforming elements is the difficulty of implementation, especially in high dimensional 
domains with high polynomial orders. In order to simplify the implementation of $C^{1}$-conforming elements, 
several kinds of non $C^{1}$-conforming numerical methods have been developed and analyzed. These include 
Morley element methods \cite{Morley1968, Wang06},  $C^{0}$- interior penalty method \cite{Brenner2005, GGN2013}, 
and DG methods \cite{Engel02, Mozolevski03}. All of these methods 
can be applied for the von K\'arm\'an equations (\ref{equ_von_karman}) 
(see \cite{Brenner2017, Carstensen19, Mallik16}). Morley element methods are very popular since these schemes 
use only six degrees of freedom on each element in two dimensional domains and don't need any stabilization 
along mesh interfaces. However, it is not straightforward to use Morley element for equations including both 
the biharmonic and Laplacian operator due to the fact that it is not $H^{1}$-conforming on the whole mesh. 
Compared with $C^{0}$ interior penalty method and DG methods,  our methods might be 
understood and implemented relatively easier by beginners, since there is no stabilization along mesh interfaces. 
All these methods can be easily modified for variable coefficient $\kappa$. 
\end{itemize}

In \cite{SZhang2018}, the biharmonic equation (\ref{equ_biharmonic}) is deduced to an equivalent system on 
three low-regularity spaces which are connected by a regular decomposition corresponding to a decomposition 
of the regularity of the high order space. A numerical method based on the equivalent system 
of (\ref{equ_biharmonic}) is presented in \cite{SZhang2018}, which can approximate solutions with low regularity 
well. But it may not be straightforward for beginners to understand and implement the method in \cite{SZhang2018}.

Our proposed mixed finite element scheme for the biharmonic equation (\ref{equ_biharmonic}) has 
the following properties. 
\begin{itemize}

\item[(1)] By using element-wise stabilization, our scheme doesn't involve any integration along mesh interfaces. 
Our scheme uses $H(\text{div})$-conforming $BDM_{k+1}$ or vector valued Lagrange element with 
order $k+1$ to approximate $\bfw = \nabla u$, and approximates $u$ with Lagrange element with 
order $k+2$ for any $k\geq 0$. Thus even beginners can relatively easily implement our scheme. 

\item[(2)] The corresponding linear system is positive definite. In fact, one method of our scheme 
produces symmetric and positive definite linear system.

\item[(3)] Our scheme can be used in arbitrary Lipschitz polyhedral domains and can approximate accurately 
solutions in $H^{2+\delta}(\Omega)$, where $\delta > \frac{1}{2}$. In fact, the numerical solution $u_{h}$ approximates 
the exact solution $u$ optimally in discrete $H^{2}$ norm and $L^{2}$ norm. We refer to 
Theorem~\ref{main_err1} and Theorem~\ref{thm_L2_con_biharmonic} for detailed description.

\item[(4)] The new method and its analysis can be generalized to nonlinear problem such as 
the von K\'arm\'an equations (\ref{equ_von_karman}). Our method (\ref{sta_method_vk}) for
the von K\'arm\'an equations  doesn't involve with any integration along mesh interfaces either.
We refer to Theorem~\ref{thm_sta_vk} and Theorem~\ref{thm_conv_von_karman} for the detailed 
description on the existence and uniqueness  of the numerical solution to the von K\'arm\'an equations 
and the optimal convergence. For the sake of  simplicity, our analysis for the von K\'arm\'an equations is 
based on the assumption that  $\Vert f\Vert_{H^{-1}(\Omega)}$ is small enough.  The success of generalization 
to the von K\'arm\'an equations is because the numerical solution $u_{h}$ of our scheme for the biharmonic 
equation satisfies the stability result:
\begin{align*}
\Vert u_{h}\Vert_{H^{1}(\Omega)} +\Vert u_{h}\Vert_{2,\calT_{h}} \leq C \Vert f\Vert_{H^{-1}(\Omega)}.
\end{align*}
Here $\Vert \cdot \Vert_{2, \calT_{h}}$ is the discrete $H^{2}$- semi norm defined in (\ref{discrete_h2}). 
With the above stability result, we can extend our analysis to isolated solutions of the von K\'arm\'an equations 
like existing works \cite{Brenner2017, Brezzi1978, Brezzi1981, Carstensen19, Mallik16, Miyoshi1976, Reinhart1982}.

\end{itemize}

In this paper, we provide a discrete $H^{2}$-norm stability (\ref{discrete_embedding_ineq2}) in 
Theorem~\ref{thm_discrete_embedding}: 
\begin{align*}
\Vert  v \Vert_{H^{1}(\Omega)}^{2} + \Vert v\Vert_{2, \calT_{h}}^{2}  
\leq C \left( \Vert \Delta v \Vert_{\calT_{h}}^{2} + \Sigma_{F\in \calE_{h}} h_{F}^{-1} \Vert \llbracket 
\nabla v \cdot \bfn \rrbracket \Vert_{0,F}^{2} \right), 
\quad \forall v \in V_{h}:= H_{0}^{1}(\Omega) \cap P_{k+2}(\calT_{h})  \text{ }  (k\geq 0), 
\end{align*}
which looks similar to the discrete Miranda–Talenti inequality \cite[($1.3$)]{Neilan2019a} (they have different 
boundary conditions). In Remark~\ref{remark_key_inequality_alternative}, we explain the fundamental difference 
between the proof of above inequality and \cite[($1.3$)]{Neilan2019a}. This inequality can help in analysis of 
the $C^{0}$ interior penalty method: to find $u_{h} \in V_{h}$ such that for any $v\in V_{h}$,
\begin{align}
\label{CIP1}
& (\kappa \Delta u_{h}, \Delta v)_{\calT_{h}} + \langle \Lb \kappa \Delta u_{h} \Rb, 
\llbracket  \nabla v \cdot  \bfn \rrbracket \rangle_{\calE_{h}}  
+ \langle \Lb \kappa \Delta v \Rb, 
\llbracket  \nabla u_{h} \cdot  \bfn \rrbracket \rangle_{\calE_{h}}  \\
\nonumber
& \qquad \qquad + \tau h^{-1}\langle \kappa \llbracket  \nabla u_{h} \cdot  \bfn \rrbracket, 
\llbracket  \nabla v \cdot  \bfn \rrbracket \rangle_{\calE_{h}} = (f, v)_{\calT_{h}}. 
\end{align}
Here, $\Lb \cdot \Rb$ and  $\llbracket   \cdot \rrbracket$ represent, respectively, the average and jump across 
the inter-element boundaries.
Obviously, (\ref{CIP1}) is a natural generalization of the $C^{0}$ interior penalty method in \cite{Brenner2005}. 
In addition, another discrete $H^{2}$-norm stability 
(\ref{discrete_embedding_ineq3}) in Theorem~\ref{thm_discrete_embedding}: 
\begin{align*}
\Vert \nabla \tilde{v}_{h} \Vert_{\calT_{h}}^{2} + \Vert \tilde{v}_{h}\Vert_{2, \calT_{h}}^{2}  
\leq C \left( \Vert \Delta \tilde{v}_{h} \Vert_{\calT_{h}}^{2} 
+ \Sigma_{F\in \calE_{h}} \big( h_{F}^{-1} \Vert \llbracket \nabla \tilde{v}_h \cdot \bfn \rrbracket \Vert_{0,F}^{2} 
+ h_{F}^{-3} \Vert \llbracket \tilde{v}_h \rrbracket \Vert_{0,F}^{2} \big) \right), 
 \forall \tilde{v}_{h} \in P_{k+2}(\calT_{h}), 
\end{align*}
will help to prove discrete $H^{2}$-norm stability of the DG methods in \cite{Mozolevski03}. 

We conclude this section with section wise description. Section 2 deals with our new  mixed type formulation.  
In section 3,  stability estimates are proved and {\it a priori} error estimates are established. 
As an application, the section 4 focuses on the von K\'arm\'an model and related error analysis using a generalization 
of the proposed method.  In Section \ref{sec:numer_example}, we give some numerical results to verify the efficiency 
of the proposed new schemes. Finally we provide a conclusion in Section \ref{sec:conclusion}.

\section{A new finite element method for the biharmonic equation}\label{sec_method_biharmonic}

This section deal with the formulation of our new mixed finite element scheme for the biharmonic equation (\ref{equ_biharmonic}). At the end of this section, we introduce the idea how to derive our scheme.

Let $\calT_h$ be the conforming triangulation of $\Omega$ made of shape-regular simplicial elements. 
We denote by $\calE_h$ the set of all faces $F$ of all elements $K\in \calT_h$, $\calE^0_h$ the set of interior 
faces of $\calT_h$, $ \calE^{\partial}_h$ the set of all faces $F$ on the boundary $\partial \Omega$, and set 
$\partial \calT_h:=\{\partial K:K\in \calT_h\}$.
For scalar-valued functions $\phi$ and $\psi$, we write
\[
 (\phi,\psi)_{\calT_h}:=\sum_{K\in\calT_h}(\phi,\psi)_K, \, \, \langle\phi,\psi\rangle_{\partial\calT_h}
 :=\sum_{K\in\calT_h}\langle\phi,\psi \rangle_{\partial K}.
\]
Here $(\cdot,\cdot)_D$ denotes the integral over
the domain $D\subset \Real^d$, and $\langle \cdot,\cdot \rangle_D$ denotes the
integral over $D\subset \Real^{d-1}$. When $D=\Omega$, we denote $(\cdot,\cdot) := (\cdot,\cdot)_\Omega$. 
For vector-valued functions, we write $(\bfphi,\bfpsi)_{\calT_h}:=\sum_{i=1}^{d} (\phi_i,\psi_i)_{\calT_h}$. 
We denote by $h_K$ the diameter of element $K \in \calT_h$ and set $h = \max_{K \in \calT_h}h_K$. 
For any face $F \in \calE_h$, $h_F$ stands for the diameter of $F$. For any interior face 
$F= \partial K \cap \partial K'$ in $\calE^0_h$, we denote by $\llbracket  \psi \rrbracket 
= (\psi|_K)|_F  - (\psi|_{K'})|_F $ the jump of a scalar function $\psi$ 
across $F$, and $\llbracket  \boldsymbol{\phi} \rrbracket = (\boldsymbol{\phi} _K \cdot \bfn_K)|_F  - 
(\boldsymbol{\phi}_{K'}\cdot \bfn_K)|_F $ the jump of a vector-valued function $\boldsymbol{\phi} $ across $F$. 
On a boundary face $F = \partial K \cap \partial \Omega$, we set $\llbracket  \psi \rrbracket = \psi$ 
and $\llbracket  \boldsymbol{\phi} \rrbracket = \boldsymbol{\phi}_K\cdot \bfn_K $.

Throughout the paper, we use the standard notations and definitions for Sobolev spaces 
(see, e.g. \cite{Adams1975}). To be more precise, let $\|\cdot\|_{s,D}$ be the usual norm on
the Sobolev space $H^s(D)$ and $\|\cdot\|_{L^p(D)}$ be the $L^p$-norm on $L^p(D)$. 
If $p=2$, we let $\|\cdot\|_{0,D}$ denote the $L^2$-norm on $L^2(D)$. Further, let
$H_{0}^{1}(\Omega):= \{v\in H^{1}(\Omega): v|_{\partial\Omega} = 0\}$, 
$H({\rm div},\Omega) := \{ \bfw \in [L^{2}(\Omega)]^{d}: \nabla\cdot \bfw \in L^{2}(\Omega) \}$ and
$H_{0}({\rm div},\Omega) := \{ \bfw \in H({\rm div},\Omega): \bfw\cdot \bfn |_{\partial\Omega} = 0\}$. 

The norm $\|\cdot\|_{\calT_h}$ is the discrete norm defined as $\|\cdot\|_{\calT_h}:= (\sum_{K \in \calT_h}
\|\cdot\|^2_{L^{2}(K)})^{\frac{1}{2}}$. We also define the discrete norm $\|\cdot\|_{ \calE_h}:= (\sum_{F \in  
\calE_h} \|\cdot\|^2_{0,F})^{\frac{1}{2}}$. For any set $\calF_h \subseteq \partial \calT_h$, we denote 
$\|h^\alpha \cdot\|_{\calF_h} := ( \sum_{F \in \calF_h}h^{2\alpha}_F \|  
\cdot \|^2_{0,F})^{\frac{1}{2}}$ with optional parameter $\alpha$. 

We also define a semi-norm $\Vert \cdot \Vert_{2,\calT_{h}}$ on $H^{2}(\calT_{h})$ as 
\begin{align}
\label{discrete_h2}
\Vert v \Vert_{2,\calT_{h}}^{2} = \Vert D^{2} v \Vert_{\calT_{h}}^{2} 
+ \Sigma_{F \in \calE_{h}} h_{F}^{-1}\Vert \llbracket \nabla v\cdot \bfn \rrbracket \Vert_{0,F}^{2}, 
\qquad \forall v \in H^{2}(\calT_{h}).
\end{align}

In this paper, $C$ denotes a positive constant depending only on the property of $\Omega$, the shape regularity of the 
meshes and the degree of polynomial spaces. The constant $C$ can take on different values in different occurrences.

In the following we present the detailed formulation of the mixed finite element scheme for (\ref{equ_biharmonic}). 
For any $k\geq 0$, we define the finite element spaces
\begin{align}
\label{FEM_spaces}
\bfW_h := H_0({\rm div},\Omega) \cap \bfP_{k+1}(\calT_h),\quad  V_h := H^1_0(\Omega) \cap P_{k+2}(\calT_h),
\end{align}
where $P_k(D)$ denotes the set of polynomials of total degree at most $k$ defined on $D$, $\bfP_k(D)$ denotes the set 
of vector-valued functions whose $d$ components lie in $P_k(D)$. We would like point out that 
an alternative choice of $\bfW_{h}$ is  
\begin{align}
\label{Wh_alternative}
\bfW_h := [H_{0}^{1}(\Omega) \cap P_{k+1}(\calT_h)]^{d}.
\end{align}
In Remark~\ref{remark_Wh}, we explain why we can use $\bfW_{h}$ in (\ref{Wh_alternative}) 
and what is its drawback. In this paper, our scheme focuses on the finite element spaces (\ref{FEM_spaces}). 

Our mixed finite element scheme is to seek an 
approximation $(\bfw_h,u_h)\in \bfW_h \times V_h$ such that for $\theta \in \{-1, 1\}$,
\begin{align}\label{sta_fem_biharmonic}
B_{\theta}((\bfw_h,u_h),(\bfeta,v)) =(f,v)_{\langle H^{-1}(\Omega), H_{0}^{1}(\Omega) \rangle}
\;\;\;\forall (\bfeta,v)\in \bfW_h \times V_h,
\end{align}
where
\begin{align}\label{def_bilinear}
B_{\theta}((\bfw_h,u_h),(\bfeta,v)) :=&
 \left(\kappa \nabla \cdot \bfw_h, \nabla \cdot \bfeta \right)_{\calT_h}  
 +\left(\nabla(\kappa \nabla \cdot \bfw_h), \bfeta - \nabla v \right)_{\calT_h} \\
 & +\theta \left( \bfw_h-\nabla u_h,  \nabla(\kappa \nabla \cdot \bfeta)   \right)_{\calT_h} 
  + \frac{\tau}{h^2}\left( \kappa (\bfw_h-\nabla u_h),  \bfeta - \nabla v \right)_{\calT_h}.\nn
\end{align}
Here, $\tau$ is a stabilization parameter to be determined later. 

\begin{rem}
\label{remark_Wh}
Notice that the boundary conditions (\ref{equ_biharmonic_dirichlet}, \ref{equ_biharmonic2}) imply that 
$\nabla u = \boldsymbol{0}$ on $\partial\Omega$. Thus the alternative $\bfW_{h}$ in (\ref{Wh_alternative}) 
is also a reasonable choice of finite element space to approximate $\nabla u$.  It can be easily shown that 
all main results (Theorem~\ref{thm_discrete_embedding}, Theorem~\ref{thm_main_sta}, Theorem~\ref{main_err1}, 
Theorem~\ref{thm_L2_con_biharmonic}, Theorem~\ref{thm_sta_vk}, Theorem~\ref{thm_conv_von_karman}) 
will still hold. The only drawback of using the alternative $\bfW_{h}$ in (\ref{Wh_alternative}) is that when 
the Dirichlet boundary conditions (\ref{equ_biharmonic_dirichlet}, \ref{equ_biharmonic2}) are not homogeneous, 
extra calculation is needed to obtain the value of $\nabla u$ on $\partial\Omega$. On the contrast, the original 
choice of $\bfW_h := H_0({\rm div},\Omega) \cap \bfP_{k+1}(\calT_h)$ can handle nonhomogeneous Dirichlet 
boundary data directly. 
\end{rem}

In the following, we always assume that
\begin{align}
\bfw = \nabla u \in \bfH^{1+\delta}(\Omega), \quad \nabla \cdot \bfw = \Delta u 
\in H^{\delta}(\Omega), \quad \delta > 1/2.
\label{reg_ass}
\end{align}

\subsection{Derivation of the finite element scheme (\ref{sta_fem_biharmonic})}

We introduce the idea how to derive the finite element scheme (\ref{sta_fem_biharmonic}) in the following. 
We temporarily assume the exact solution $u$ of (\ref{equ_biharmonic}) and $\kappa$ are smooth. 

For any $v \in V_{h}$, we have 
\begin{align*}
(\Delta (\kappa \Delta u), v)_{\Omega} = (f, v)_{\Omega}. 
\end{align*}
Doing integration by parts, we obtain 
\begin{align*}
- (\nabla (\kappa \Delta u), \nabla v)_{\Omega} = (f, v)_{\Omega},
\end{align*}
since $v = 0$ on $\partial\Omega$. 
Obviously, if we want to do integration by parts one more time, we will encounter terms along mesh 
interfaces since $\nabla v$ is not $H(\text{div})$-conforming. We take $\bfeta \in \bfW_h$ arbitrarily. 
We would like to ``replace" $\nabla v$ by $\bfeta$ such that integration by parts can be done. So we have 
\begin{align*}
- (\nabla (\kappa \Delta u), \bfeta)_{\Omega}- (\nabla (\kappa \Delta u), \nabla v - \bfeta)_{\Omega} 
= (f, v)_{\Omega}.
\end{align*}
Now we can do integration by parts to the first term in the left hand side of above equation. 
Thus we have 
\begin{align*}
(\kappa \Delta u, \nabla \cdot \bfeta)_{\Omega} - (\nabla (\kappa \Delta u), \nabla v - \bfeta)_{\Omega} 
= (f, v)_{\Omega}.
\end{align*}
The above equation inspire us to use $(\bfw_{h}, u_{h}) \in \bfW_{h} \times V_{h}$ to approximate 
$(\nabla u, u)$:  
\begin{align*}
(\kappa\nabla\cdot \bfw_{h}, \nabla \cdot \bfeta)_{\Omega} - (\nabla (\kappa \nabla \cdot \bfw_{h}), 
\nabla v - \bfeta)_{\calT_{h}} = (f, v)_{\Omega}, \quad \forall (\bfeta, v) \in \bfW_{h} \times V_{h}.
\end{align*} 
In order to have well-posedness, we need $\bfw_{h}$ and $\nabla u_{h}$ are ``close" enough to each other. 
So we need to add stabilization into the above equation to have: 
\begin{align*}
(\kappa\nabla\cdot \bfw_{h}, \nabla \cdot \bfeta)_{\Omega} +(\nabla ( \kappa \nabla \cdot \bfw_{h}), 
\bfeta - \nabla v)_{\calT_{h}} + \frac{\tau}{h^2}\left( \kappa (\bfw_h-\nabla u_h),  \bfeta - \nabla v \right)_{\calT_h}
= (f, v)_{\Omega}, \quad \forall (\bfeta, v) \in \bfW_{h} \times V_{h}.
\end{align*} 
Here $\tau$ is a positive constant. In order to have a symmetric or anti-symmetric method, we 
add the term $\theta \left( \bfw_h-\nabla u_h,  \nabla(\kappa\nabla \cdot \bfeta)   \right)_{\calT_h}$ 
to the above equation to get the finite element scheme (\ref{sta_fem_biharmonic}). 

\section{Analysis of the mixed finite element scheme for the biharmonic equation}
In this section, we discuss the stability and error estimates of the mixed finite 
element scheme (\ref{sta_fem_biharmonic}) for the biharmonic equation (\ref{equ_biharmonic}).

\subsection{Stability estimate}

\begin{thm}
\label{thm_discrete_embedding}
There exists a positive constant $C$ such that for any $(\bfeta_{h}, v_{h}) \in \bfW_{h}\times V_{h}$, 
\begin{subequations}
\label{discrete_embedding_ineqs}
\begin{align}
\label{discrete_embedding_ineq1}
& \Vert  v_{h} \Vert_{H^{1}(\Omega)}^{2} + \Vert v_{h}\Vert_{2, \calT_{h}}^{2} 
\leq C \left( \Vert \nabla \cdot \bfeta_{h} \Vert_{\calT_{h}}^{2} 
+ \Sigma_{K \in \calT_{h}} h_{K}^{-2} \Vert \bfeta_{h} - \nabla v_{h} \Vert_{L^{2}(K)}^{2} \right), 
\quad \forall (\bfeta_{h}, v_{h}) \in \bfW_{h}\times V_{h}; \\
\label{discrete_embedding_ineq2}
& \Vert  v_{h} \Vert_{H^{1}(\Omega)}^{2} + \Vert v_{h}\Vert_{2, \calT_{h}}^{2}  
\leq C \left( \Vert \Delta v_{h} \Vert_{\calT_{h}}^{2} 
+ \Sigma_{F\in \calE_{h}} h_{F}^{-1} \Vert \llbracket \nabla v_h \cdot \bfn \rrbracket \Vert_{0,F}^{2} \right), 
\quad \forall v_{h} \in V_{h}; \\
\label{discrete_embedding_ineq3}
& \Vert \nabla \tilde{v}_{h} \Vert_{\calT_{h}}^{2} + \Vert \tilde{v}_{h}\Vert_{2, \calT_{h}}^{2}  
\leq C \left( \Vert \Delta \tilde{v}_{h} \Vert_{\calT_{h}}^{2} 
+ \Sigma_{F\in \calE_{h}} \big( h_{F}^{-1} \Vert \llbracket \nabla \tilde{v}_h \cdot \bfn \rrbracket \Vert_{0,F}^{2} 
+ h_{F}^{-3} \Vert \llbracket \tilde{v}_h \rrbracket \Vert_{0,F}^{2} \big) \right), 
 \forall \tilde{v}_{h} \in P_{k+2}(\calT_{h}).
\end{align}
\end{subequations}
\end{thm}

\begin{rem}
\label{remark_key_inequality_alternative}
Though (\ref{discrete_embedding_ineq2}) looks similar to the discrete Miranda–Talenti inequality 
\cite[($1.3$)]{Neilan2019a} (they have different boundary conditions), their proofs are fundamentally 
different from ours. An enriching operator \cite[Lemma~$3$]{Neilan2019a} for $H^{2}$-conforming Clough–Tocher 
elements is used to obtain \cite[($1.3$)]{Neilan2019a}. Since there is no Clough–Tocher elements with 
polynomial order greater than $3$ in three or higher dimensional domains, \cite[($1.3$)]{Neilan2019a} 
is valid for polynomial order less than $3$ in three dimensional domains. If we mimic the methodology in 
\cite{Neilan2019a}, our result (\ref{discrete_embedding_ineq2}) should have the same restriction 
on polynomial order as \cite[($1.3$)]{Neilan2019a}. However, our proof of (\ref{discrete_embedding_ineq2}) 
treats $\nabla v_{h}$ as $1$-form and uses standard enriching operator for vector valued $H^{1}$-conforming 
elements. Thus in (\ref{discrete_embedding_ineq2}), there is no restriction on the dimension of domains and 
order of polynomial.    
\end{rem}

\begin{proof}
By triangle inequality and discrete inverse inequality, 
\begin{align*}
 \Vert \Delta v_{h} \Vert_{\calT_{h}}^{2} 
\leq & C \left( \Vert \nabla\cdot (\bfeta_{h} - \nabla v_{h})  \Vert_{\calT_{h}}^{2}
 + \Vert \nabla\cdot \bfeta_{h} \Vert_{\calT_{h}}^{2} \right) \\
 \leq & C \left( \Vert \nabla \cdot \bfeta_{h} \Vert_{\calT_{h}}^{2} 
+ \Sigma_{K \in \calT_{h}} h_{K}^{-2} \Vert \bfeta_{h} - \nabla v_{h} \Vert_{L^{2}(K)}^{2} \right). 
\end{align*}
By discrete trace inequality and the fact that $\bfeta_{h} \in H_{0}(\text{div},\Omega)$, 
\begin{align*}
\sum_{F\in \calE_{h}} h_{F}^{-1} \Vert \llbracket \nabla v_h \cdot \bfn \rrbracket  \Vert_{0,F}^{2} 
= \sum_{F\in \calE_{h}} h_{F}^{-1} \Vert \llbracket (\bfeta - \nabla v_h ) \cdot \bfn \rrbracket  \Vert_{0,F}^{2} 
\leq C \Sigma_{K \in \calT_{h}} h_{K}^{-2} \Vert \bfeta_{h} - \nabla v_{h} \Vert_{L^{2}(K)}^{2}.
\end{align*}
Combing the above two inequalities, we have 
\begin{align}
\label{thm_discrete_embedding_tmp_ineq1}
\Vert \Delta v_{h} \Vert_{\calT_{h}}^{2} + \sum_{F\in \calE_{h}} h_{F}^{-1} \Vert \llbracket \nabla v_h \cdot 
\bfn \rrbracket  \Vert_{0,F}^{2} \leq C \left( \Vert \nabla \cdot \bfeta_{h} \Vert_{\calT_{h}}^{2} 
+ \Sigma_{K \in \calT_{h}} h_{K}^{-2} \Vert \bfeta_{h} - \nabla v_{h} \Vert_{L^{2}(K)}^{2} \right). 
\end{align}

Since $v_{h} \in H_{0}^{1}(\Omega)$, the tangential components of $\nabla v_{h}$ are continuous across 
interior mesh interfaces and vanish along the boundary $\partial\Omega$. Thus 
\begin{align*}
\sum_{F\in \calE_{h}} h_{F}^{-1} \Vert \llbracket \nabla v_h \cdot \bfn \rrbracket  \Vert_{0,F}^{2} 
= \sum_{F\in \calE_{h}} h_{F}^{-1} \Vert \llbracket \nabla v_h \rrbracket  \Vert_{0,F}^{2}, 
\end{align*}
where $\llbracket \nabla v_h \rrbracket|_{F}$ is the jump of all components of $\nabla v_{h}$ across  
interior mesh interface $F$ and $\llbracket \nabla v_h \rrbracket|_{F} = \nabla v_{h}|_{F}$ for 
any face $F$ on $\partial\Omega$. According to \cite[Theorem~$2.2$]{KP2003} and the above equality, 
there is $\tilde{\bfeta}_{h} \in [ H_{0}^{1}(\Omega) \cap P_{k+1}(\calT_{h}) ]^{d}$ such that 
\begin{align}
\label{thm_discrete_embedding_tmp_ineq2}
\Vert \tilde{\bfeta}_{h} - \nabla v_{h}\Vert_{\calT_{h}}^{2} \leq C \sum_{F\in \calE_{h}} h_{F}
\Vert \llbracket \nabla v_h \rrbracket  \Vert_{0,F}^{2}  \leq C \sum_{F\in \calE_{h}} h_{F}
\Vert \llbracket \nabla v_h \cdot \bfn \rrbracket  \Vert_{0,F}^{2} . 
\end{align}
We would like to point out that though it only deals with $d=2,3$, the proof of \cite[Theorem~$2.2$]{KP2003} 
can be easily generalized for arbitrary dimension $d \in \mathbb{N}$. Thus 
(\ref{thm_discrete_embedding_tmp_ineq2}) is valid for any dimension  $d \in \mathbb{N}$.

Obviously, $\tilde{\bfeta}_{h}$ is a $1$-form on $\Omega$ and therefore (cf. \cite{Arnold2006}), 
\begin{align*}
- \Delta \tilde{\bfeta}_{h} = \boldsymbol{d}_{0} \boldsymbol{d}_{1}^{*}\tilde{\bfeta}_{h} 
+ \boldsymbol{d}_{2}^{*} \boldsymbol{d}_{1} \tilde{\bfeta}_{h} \in H^{-1} \Lambda^{1}(\Omega), 
\end{align*}
where for any $1\leq k \leq d$, $\boldsymbol{d}_{k}$ is the exterior derivative mapping from $k$-form to 
$(k+1)$-form and $\boldsymbol{d}_{k}^{*} = (-1)^{d(k+1)+1} *\boldsymbol{d}_{d-k} *$ maps from $k$-form 
to $(k-1)$-form. Here $*$ is the Hodge star operator. 
 In fact, $\boldsymbol{d}_{0} = \nabla$ and $\boldsymbol{d}_{1}^{*} = - \nabla\cdot$. 
If $d=3$, $\boldsymbol{d}_{1} = \boldsymbol{d}_{2}^{*} = \nabla \times$. 
Then for any $\boldsymbol{\phi} \in C_{0}^{\infty}\Lambda^{1}(\Omega)$, 
\begin{align*}
 - (\Delta \tilde{\bfeta}_{h}, \boldsymbol{\phi})_{\Omega}  
= (\boldsymbol{d}_{1}^{*} \tilde{\bfeta}_{h}, \boldsymbol{d}_{1}^{*}\boldsymbol{\phi})_{\Omega} 
+ (\boldsymbol{d}_{1}\tilde{\bfeta}_{h}, \boldsymbol{d}_{1}\boldsymbol{\phi})_{\Omega}.
\end{align*}
Since $\Vert \boldsymbol{d}_{1}^{*}\boldsymbol{\phi} \Vert_{L^{2}(\Omega)} \leq 
C \Vert \boldsymbol{\phi}\Vert_{H^{1}(\Omega)}$ and $\Vert \boldsymbol{d}_{1}\boldsymbol{\phi} 
\Vert_{L^{2}(\Omega)} \leq C \Vert \boldsymbol{\phi} \Vert_{H^{1}(\Omega)}$, the above identity implies 
\begin{align*}
\Vert \Delta \tilde{\bfeta}_{h} \Vert_{H^{-1}(\Omega)} \leq C \left( \Vert \boldsymbol{d}_{1}^{*} 
\tilde{\bfeta}_{h}\Vert_{L^{2}(\Omega)} + \Vert \boldsymbol{d}_{1}\tilde{\bfeta}_{h}\Vert_{L^{2}(\Omega)} \right).
\end{align*}
Since $\tilde{{\bfeta}_{h}} = \boldsymbol{0}$ on $\partial\Omega$, the above inequality implies 
\begin{align*}
\Vert \tilde{\bfeta}_{h} \Vert_{H^{1}(\Omega)}  \leq C \left( \Vert \boldsymbol{d}_{1}^{*} 
\tilde{\bfeta}_{h}\Vert_{L^{2}(\Omega)} + \Vert \boldsymbol{d}_{1}\tilde{\bfeta}_{h}\Vert_{L^{2}(\Omega)} \right).
\end{align*}
By triangle inequality and the fact that $\boldsymbol{d}_{0} = \nabla$ and 
$\boldsymbol{d}_{1}^{*} = - \nabla\cdot$, we arrive at 
\begin{align}
\label{thm_discrete_embedding_tmp_ineq3}
& \Vert \tilde{\bfeta}_{h} \Vert_{H^{1}(\Omega)}  \leq C \left( \Vert \boldsymbol{d}_{1}^{*} 
\tilde{\bfeta}_{h}\Vert_{L^{2}(\Omega)} + \Vert \boldsymbol{d}_{1}\tilde{\bfeta}_{h}\Vert_{L^{2}(\Omega)} \right) \\
\nonumber
 \leq &  C \left( \Vert \boldsymbol{d}_{1}^{*} 
(\tilde{\bfeta}_{h} - \nabla v_{h})\Vert_{\calT_{h}} + \Vert \boldsymbol{d}_{1}(\tilde{\bfeta}_{h} - \nabla v_{h})
\Vert_{\calT_{h}} + \Vert  \boldsymbol{d}_{1}^{*} (\nabla v_{h}) \Vert_{\calT_{h}} 
+ \Vert  \boldsymbol{d}_{1} (\nabla v_{h}) \Vert_{\calT_{h}}  \right)  \\ 
\nonumber
= & C \left( \Vert \boldsymbol{d}_{1}^{*} 
(\tilde{\bfeta}_{h} - \nabla v_{h})\Vert_{\calT_{h}} + \Vert \boldsymbol{d}_{1}(\tilde{\bfeta}_{h} - \nabla v_{h})
\Vert_{\calT_{h}} + \Vert \Delta v_{h} \Vert_{\calT_{h}}  \right) \\
\nonumber
\leq & C \left( \Sigma_{K \in \calT_{h}} h_{K}^{-1} \Vert \tilde{\bfeta}_{h} - \nabla v_{h}\Vert_{\calT_{h}} 
 + \Vert \Delta v_{h} \Vert_{\calT_{h}}  \right). 
\end{align}
We utilized discrete inverse inequality to obtain the above inequality. 

By (\ref{thm_discrete_embedding_tmp_ineq2}) and (\ref{thm_discrete_embedding_tmp_ineq3}), we obtain 
(\ref{discrete_embedding_ineq2}). (\ref{discrete_embedding_ineq1}) is an immediate application of 
(\ref{discrete_embedding_ineq2}) and (\ref{thm_discrete_embedding_tmp_ineq1}).  

Now we want to prove (\ref{discrete_embedding_ineq3}). 
Like (\ref{thm_discrete_embedding_tmp_ineq2}), there is $v_{h} \in 
H_{0}^{1}(\Omega) \cap P_{k+2}(\calT_{h})$ such that 
\begin{align*}
\Vert \tilde{v}_{h} - v_{h}\Vert_{\calT_{h}}^{2} \leq C \Sigma_{F \in \calE_{h}} 
h_{F}\Vert \llbracket  v_h \rrbracket  \Vert_{0,F}^{2}.  
\end{align*}
The above inequality with (\ref{discrete_embedding_ineq2}) yields (\ref{discrete_embedding_ineq3}) 
and this concludes the rest of the proof.
\end{proof}
  
Below,  we state the stability estimate for the mixed finite element scheme (\ref{sta_fem_biharmonic}).
\begin{thm}\label{thm_main_sta}
For the solution $(\bfw_h,u_h)$ of the mixed finite element scheme (\ref{sta_fem_biharmonic}) with
$\theta = 1$ if the stabilization parameter $\tau$ is chosen to be large enough, then, there exists 
a positive constant $C$ independent of $h$ and penalty parameter $\tau$
 such that
\begin{align}
\label{main_sta_est}
\|\bfw_h\|_{H({\rm div},\Omega)}+ \|u_{h}\|_{H^{1}(\Omega)} + \Vert u_{h}\Vert_{2,\calT_{h}} 
\leq C \|f\|_{H^{-1}(\Omega)}.
\end{align}
When $\theta=-1$, the stability estimate (\ref{main_sta_est}) holds for any $\tau>0$. 
\end{thm}

\begin{proof}
Choose $\bfeta=\bfw_h$ and $v = u_h$ in (\ref{sta_fem_biharmonic}) to obtain for $\theta\neq -1$
\begin{align}
\label{est_sta_1}
\| \sqrt{\kappa} \nabla \cdot \bfw_h\|^2_{\calT_h} + \frac{\tau}{h^2}\;\|  \sqrt{\kappa} 
(\bfw_h-\nabla u_h)\|^2_{\calT_h} = 
(f,u_{h})_{\langle H^{-1}(\Omega), H_{0}^{1}(\Omega) \rangle}
- (1+\theta)\;\left(\nabla(\kappa \nabla \cdot \bfw_h), \bfw_h - \nabla u_h \right)_{\calT_h}.
\end{align}
For the second term on the right hand side of (\ref{est_sta_1}), a use of the  Cauchy-Schwarz inequality 
with the  inverse inequality and the fact that $\kappa \in W^{1, \infty}(\Omega)$  yields
\begin{align}
\label{est_sta_2}
(1+\theta) \left(\nabla(\kappa\nabla \cdot \bfw_h), \bfw_h - \nabla u_h \right)_{\calT_h}  
&\leq C \| \nabla(\kappa \nabla \cdot \bfw_h) \|_{\calT_h} \|\bfw_h - \nabla u_h\|_{\calT_h}\\
\nonumber
&\leq C_s \| \nabla \cdot \bfw_h\|_{\calT_h} h^{-1}\|\bfw_h - \nabla u_h\|_{\calT_h}\\
\nonumber
& \leq  \frac{1}{2}  \| \nabla \cdot \bfw_h\|^2_{\calT_h} + C_s^2  h^{-2}\|\bfw_h - \nabla u_h\|^2_{\calT_h}. 
\end{align}
Substituting (\ref{est_sta_2}) in (\ref{est_sta_1}),  choose $\tau$ large enough such that $\tau > C_s^2.$
Then, we obtain
\begin{align}
\label{est_sta_3}
 \| \nabla \cdot \bfw_h\|^2_{\calT_h}  + h^{-2}\|\bfw_h - \nabla u_h\|^2_{\calT_h} 
 \leq C \|f\|_{H^{-1}(\Omega)} \| \nabla u_h\|_{\calT_h}. 
\end{align}
From Theorem~\ref{thm_discrete_embedding}, it follows that
\begin{align}
\label{est_sta_4}
\Vert u_{h} \Vert_{H^{1}(\Omega)}^{2} +  \Vert u_{h}\Vert^2_{2, \calT_{h}}
& \leq C \left(  h^{-2}\|\bfw_h - \nabla u_h\|^2_{\calT_h}+ \| \nabla\cdot \bfw_h \|^2_{\calT_h}  \right). 
\end{align}
Combining (\ref{est_sta_3}, \ref{est_sta_4}), we  obtain
\begin{align}
\label{est_sta_5}
\|\nabla \cdot \bfw_h\|^2_{\calT_h} + \Vert u_{h} \Vert_{H^{1}(\Omega)}^{2} 
 + \Vert u_{h}\Vert^2_{2, \calT_{h}} \leq C \|f\|_{H^{-1}(\Omega)}\| \nabla u_h\|_{\calT_h}.
\end{align}
Now an application of the Cauchy-Schwarz inequality with (\ref{est_sta_5}) yields 
the part of estimate (\ref{main_sta_est}). To complete the rest of the estimate for for $\theta = 1,$
we note from (\ref{est_sta_3}) that 
\begin{align} \label{w-u-1}
\|\bfw_h - \nabla u_h\|_{\calT_h}  \leq C h \|f\|_{H^{-1}(\Omega)}.
\end{align}
Since 
$$\|\bfw_h\|_{\calT_h}\leq \|\bfw_h - \nabla u_h\|_{\calT_h} +  \|\nabla u_h\|_{\calT_h},$$
(\ref{est_sta_5}, \ref{w-u-1}) show the following 
stability estimate for $\bfw_h$ in $L^2$ norm as 
$$
\|\bfw_h\|_{\calT_h}  \leq C \|f\|_{H^{-1}(\Omega)}.
$$
This concludes the desired result  for $\theta = 1$.  When $\theta=-1$, 
the second term on the right hand side of (\ref{est_sta_1}) becomes 
zero and the rest of the proof follows as above for any $\tau>0.$ This completes rest of the proof.
\end{proof}

\subsection{Error estimates}\label{sec_err_b}
In this section, we present the detailed proof of the {\it a priori} error estimates for the 
mixed finite element scheme (\ref{sta_fem_biharmonic}). 

Now define
\[
\bfe_{\bfw} :=  \bfPi^{\rm div}_h\bfw - \bfw_h, \quad e_u :=  \pi_h u -u_h,
\]
where $\bfPi^{\rm div}_h: H_{0}({\rm div},\Omega) \rightarrow \bfW_h$ is the $H({\rm div})$-smooth projection and 
$\pi_h: H^1_0(\Omega)\rightarrow V_h$ is the $H^1$-smooth projection introduced in \cite{ChrisWinth08, Schob08b} 
(the interpolations in \cite{DB2005} can be used also). By the regularity assumption (\ref{reg_ass}), the following 
approximation properties hold true for the two projections $\bfPi^{\rm div}_h$ and $\pi_h$:
\begin{subequations}
\label{appros}
\begin{align}
\label{appro_1}
 \| \bfw - \bfPi^{\rm div}_h\bfw\|_{H({\rm div},\Omega)}  & + h^{\frac{1}{2}}\| \nabla \cdot( \bfw - \bfPi^{\rm div}_h\bfw  )
 \|_{\partial \calT_h} \leq C h^{s} \left( \| \bfw \|_{s,\Omega} + \| \nabla\cdot \bfw \|_{s,\Omega}\right), \\ 
 \label{appro_2}
& \| \bfw - \bfPi^{\rm div}_h\bfw\|_{\calT_h} + \| \nabla ( u - \pi_h u) \|_{\calT_h}  \leq Ch^{1+s} \|\bfw \|_{1+s,\Omega}, \\
 \label{appro_3}
& \Vert u - \pi_h u \Vert_{2, \calT_{h}}  \leq C h^{s} \| \bfw \|_{1+s,\Omega}, 
\end{align}
\end{subequations}
where $s =\min\{\delta, k+1\}$, $u$ is the solution of (\ref{equ_biharmonic}) and $\bfw = \nabla u$.

We are now ready to present the error equation for our subsequent  error analysis.
\begin{lem}\label{lem_err_bih}
Let $u$ and $(\bfw_h,u_h)$ be the solution of (\ref{equ_biharmonic}) and (\ref{sta_fem_biharmonic}), 
with $\bfw = \nabla u,$ respectively. Under the regularity assumption in (\ref{reg_ass}), there holds
\begin{align}\label{equ_err}
& B_{\theta} ( (\bfe_{\bfw},e_u), (\bfeta, v)) \\
& = -(\kappa \nabla \cdot( \bfw -  \bfPi^{\rm div}_h \bfw ),\nabla \cdot \bfeta )_{\calT_h}  
- \tau h^{-2} ( \bfw -  \bfPi^{\rm div}_h \bfw  - \nabla( u-\pi_h u ), \kappa(\bfeta - \nabla v)  )_{\calT_h} \nn\\
& \quad +( \kappa\nabla \cdot(\bfw -\bfPi^{\rm div}_h \bfw ) , \nabla \cdot(\bfeta - \nabla v) )_{\calT_h} 
- \langle \kappa\nabla \cdot (\bfw -\bfPi^{\rm div}_h \bfw), (\bfeta - \nabla v) \cdot 
\bfn \rangle_{\partial \calT_h}\nn\\
&\quad -\theta( \bfw -  \bfPi^{\rm div}_h \bfw -  \nabla( u-\pi_h u ), 
\nabla (\kappa\nabla \cdot \bfeta))_{\calT_h},\nn
\end{align}
for any $(\bfeta,v) \in \bfW_h \times V_h$.
\end{lem}
\begin{proof}
We assume $\delta \in (\frac{1}{2}, 1]$ where $\delta$ is introduced in (\ref{reg_ass}). 
We choose $(\bfeta, v)\in \bfW_h \times V_h$ arbitrarily.

We define $\tilde{v}$ to be a function on $\mathbb{R}^{d}$ satisfying 
\begin{align*}
\tilde{v}(\boldsymbol{x}) & = v(\boldsymbol{x}) \qquad & \forall \boldsymbol{x} \in \Omega, \\
\tilde{v}(\boldsymbol{x}) & = 0 \qquad & \forall \boldsymbol{x} \in \mathbb{R}^{d} \setminus \Omega.
\end{align*}
Obviously, $\tilde{v} \in H^{1}(\mathbb{R}^{d})$. 
It is easy to verify that $\nabla \tilde{v} \in [ H^{1- \delta}(\mathbb{R}^{d}) ]^{d}$. 
Then, according to \cite[Theorem~$3.33$]{Mclean1}, we have that $v \in H^{2 - \delta}_{0}(\Omega)$ 
where $H^{2 - \delta}_{0}(\Omega)$ is the closure of $C_{0}^{\infty}(\Omega)$ with respect to 
$\Vert \cdot \Vert_{H^{2 - \delta}(\Omega)}$ (the standard norm of $H^{2-\delta}(\Omega)$). 
In addition, it is easy to verify that $\bfeta \in [ H^{1 - \delta}(\Omega) ]^{d}$.  

Since $\kappa \in W^{1,\infty}(\Omega)$ and $\nabla\cdot \bfw \in H^{\delta}(\Omega)$, 
then $\kappa \nabla\cdot \bfw \in H^{\delta}(\Omega)$. 
Thus, $-(\kappa\nabla\cdot \bfw, \nabla \cdot \bfeta)_{\calT_{h}} 
+ \langle \kappa\nabla \cdot \bfw, (\bfeta - \nabla v)\cdot \bfn \rangle_{\partial \calT_{h}}$ 
is well defined. By \cite[Theorem~$3.40$]{Mclean1}, $H^{1 - \delta}(\Omega) = H_{0}^{1 - \delta}(\Omega)$. 
Since $\bfeta, \nabla v \in [ H^{1 - \delta}(\Omega) ]^{d}$, 
$\langle \nabla (\kappa\nabla\cdot \bfw), \bfeta - \nabla v\rangle_{\Omega}$ is well defined 
where $\langle \nabla (\kappa\nabla\cdot \bfw), \bfeta - \nabla v\rangle_{\Omega}$ is the coupling 
between $[H^{\delta - 1}(\Omega)]^{d}$ and $[H_{0}^{1 - \delta}(\Omega)]^{d}$.
	
By (\ref{equ_biharmonic1}) and the fact that $f \in H^{-1}(\Omega)$, there holds
\begin{align*}
(\Delta (\kappa\nabla\cdot \bfw), v)_{\calT_{h}} = (f, v)_{\langle H^{-1}(\Omega), H_{0}^{1}(\Omega) \rangle}.
\end{align*}
We notice that for any $\bar{v}\in C_{0}^{\infty}(\Omega)$, 
\begin{align*}
(\Delta (\kappa\nabla\cdot \bfw), \bar{v})_{\calT_{h}} = 
- \left(\nabla (\kappa \nabla\cdot \bfw), \nabla \bar{v}\right)_{\Omega}.
\end{align*}
Since $C_{0}^{\infty}(\Omega)$ is dense in $H^{2 - \delta}_{0}(\Omega)$, then for 
$v\in H^{2 - \delta}_{0}(\Omega)$, it follows that 
\begin{align*}
(\Delta (\kappa\nabla\cdot \bfw), v)_{\calT_{h}} = - \left(\nabla (\kappa\nabla\cdot \bfw), \nabla v\right)_{\Omega} 
= -\left(\nabla (\kappa\nabla\cdot \bfw), \bfeta \right)_{\Omega} 
+ \left(\nabla (\kappa\nabla\cdot \bfw), \bfeta - \nabla v\right)_{\Omega} .
\end{align*}
We recall that $\kappa \nabla\cdot \bfw \in H^{\delta}(\Omega)$ since $\kappa \in W^{1,\infty}(\Omega)$. 
Let $\{ \sigma_{i} \}_{i=1}^{\infty} \subset D(\bar{\Omega}):= \{ \sigma: \sigma = \bar{\sigma}|_{\Omega} 
\text{ where } \bar{\sigma} \in C_{0}^{\infty}(\mathbb{R}^{d}) \}$ such that 
\begin{align*}
\Vert \sigma_{i} - \kappa\nabla\cdot \bfw \Vert_{H^{\delta}(\Omega)} \rightarrow 0 \text{ as } i\rightarrow \infty.
\end{align*}
Then from  \cite[Theorem~$1.4.4.6$]{Grisvard1}, we arrive at 
\begin{align}
\label{negative_appr1}
\Vert \nabla (\sigma_{i} - \kappa\nabla\cdot \bfw) \Vert_{H^{\delta-1}(\Omega)} 
\rightarrow 0 \text{ as } i\rightarrow \infty.
\end{align}
For any $i \geq 1$, 
\begin{align*}
-\left(\nabla \sigma_{i}, \bfeta \right)_{\Omega} + \left(\nabla \sigma_{i}, \bfeta - \nabla v\right)_{\Omega} 
= ( \sigma_{i}, \nabla \cdot \bfeta)_{\calT_{h}} - (\sigma_{i}, \nabla\cdot (\bfeta - \nabla v))_{\calT_{h}} 
+ \langle \sigma_{i}, (\bfeta - \nabla v) \cdot \bfn \rangle_{\partial \calT_{h}}.
\end{align*}
Then, letting $i\mapsto \infty$, we apply  (\ref{negative_appr1}) and  integration by parts to obtain
\begin{align*}
(\Delta (\kappa\nabla\cdot \bfw), v)_{\calT_{h}} = (\kappa\nabla\cdot \bfw, \nabla \cdot \bfeta)_{\calT_{h}}
 - (\kappa\nabla\cdot \bfw, \nabla\cdot (\bfeta - \nabla v))_{\calT_{h}} 
+ \langle \kappa\nabla\cdot \bfw, (\bfeta - \nabla v) \cdot \bfn \rangle_{\partial \calT_{h}}
\end{align*}
Thus, we arrive at 
\begin{align}
\label{consistent_biharmonic}
 (\kappa\nabla\cdot \bfw, \nabla \cdot \bfeta)_{\calT_{h}} 
 - (\kappa\nabla\cdot \bfw, \nabla\cdot (\bfeta - \nabla v))_{\calT_{h}} 
+ \langle \kappa \nabla\cdot \bfw, (\bfeta - \nabla v) \cdot \bfn \rangle_{\partial \calT_{h}} 
= (f, v)_{\langle H^{-1}(\Omega), H_{0}^{1}(\Omega) \rangle}.
\end{align}
By subtracting (\ref{consistent_biharmonic}) from (\ref{sta_fem_biharmonic}), we immediately obtain 
the error equation (\ref{equ_err}) and this concludes the proof.
\end{proof}

Below, we present the main theorem of this section.
\begin{thm}\label{main_err1}
Let $u$ and $(\bfw_h,u_h)$ be the solution of (\ref{equ_biharmonic}) and (\ref{sta_fem_biharmonic}),
respectively, and $\bfw = \nabla u$. Further let the regularity assumption in (\ref{reg_ass}) hold. 
For $\theta = 1,$  if the stabilization parameter $\tau$ is chosen to be large enough, then there holds
\begin{align*}
\| \nabla \cdot (\bfw - \bfw_h) \|_{\calT_h} + \Vert u - u_{h} \Vert_{H^{1}(\Omega)}
+ \Vert u - u_{h}\Vert_{2, \calT_{h}}
\leq C h^s\| \bfw \|_{1+s,\Omega},
\end{align*}
where $s= \min\{\delta, k+1\}$ and $\delta$ is the parameter given in (\ref{reg_ass}).
When $\theta=-1,$ the above estimate holds for  $\tau>0.$
\end{thm}
\begin{proof}
Choose $(\bfeta,v) = (\bfe_{\bfw},e_u)$ in  (\ref{equ_err}) to obtain
\begin{align*}
\|\nabla \cdot \bfe_{\bfw} \|^2_{\calT_h} + \tau h^{-2} \|  \bfe_{\bfw} -  \nabla e_u \|^2_{\calT_h}  =  \sum^5_{k=0}T_k,
\end{align*}
where 
\begin{align*}
T_0 &=- (1+\theta) ( \nabla( \kappa\nabla \cdot \bfe_{\bfw}), \bfe_{\bfw} -  \nabla e_u)_{\calT_h} ,\\
T_1 &= -( \kappa\nabla \cdot( \bfw -  \bfPi^{\rm div}_h \bfw ),\nabla \cdot \bfe_{\bfw} )_{\calT_h}, \\
T_2 &=  - \tau h^{-2} ( \bfw -  \bfPi^{\rm div}_h \bfw  - \nabla( u-\pi_h u ), 
\kappa(\bfe_{\bfw} - \nabla e_u ) )_{\calT_h}, \\
T_3 &= ( \kappa\nabla \cdot(\bfw -\bfPi^{\rm div}_h \bfw ) , \nabla \cdot(\bfe_{\bfw} - \nabla e_u) )_{\calT_h} ,\\
T_4 &= - \langle \kappa\nabla \cdot (\bfw -\bfPi^{\rm div}_h \bfw), (\bfe_{\bfw} - \nabla e_u) \cdot 
\bfn \rangle_{\partial \calT_h},\\
T_5 &=  -\theta ( \bfw -  \bfPi^{\rm div}_h \bfw -  \nabla( u-\pi_h u ), 
\nabla (\kappa\nabla \cdot \bfe_{\bfw}))_{\calT_h}.
\end{align*}
By the inverse inequality, the approximation properties of the two operators $\bfPi^{\rm div}_h$, $\pi_h$ 
in (\ref{appros}), and the fact that $\kappa \in W^{1,\infty}(\Omega)$, 
we can derive the upper bounds of $T_0,\cdots,T_5$ as follow:
\begin{align*}
T_0 & \leq C\| \nabla \cdot \bfe_{\bfw} \|_{\calT_h} h^{-1} \|\bfe_{\bfw} -  \nabla e_u\|_{\calT_h},\\
T_1 &\leq C h^s\| \bfw \|_{1+s,\Omega} \| \nabla \cdot \bfe_{\bfw} \|_{\calT_h},\\
T_2 & \leq C h^s\| \bfw \|_{1+s,\Omega} h^{-1}\|\bfe_{\bfw} -  \nabla e_u\|_{\calT_h},\\
T_3 & \leq C h^s\| \bfw \|_{1+s,\Omega}h^{-1}\|\bfe_{\bfw} -  \nabla e_u\|_{\calT_h},\\
T_4 & \leq C h^{\frac{1}{2}} \|\nabla \cdot (\bfw -\bfPi^{\rm div}_h \bfw)\|_{\partial \calT_h} h^{-1}\|\bfe_{\bfw} -  \nabla e_u\|_{\calT_h} \quad (\text{by trace inequality})\\
&\leq C h^s\| \bfw \|_{1+s,\Omega}h^{-1}\|\bfe_{\bfw} -  \nabla e_u\|_{\calT_h},\\
T_5 & \leq C h^s\| \bfw \|_{1+s,\Omega}  \| \nabla \cdot \bfe_{\bfw} \|_{\calT_h}.
\end{align*}
Now combining the above estimates for $T_k,k=0,\cdots,5$, the Cauchy-Schwarz inequality, and choosing the 
stabilization parameter $\tau$ to be large enough, we arrive at
\begin{align}
\label{err_res1}
\|\nabla \cdot \bfe_{\bfw} \|_{\calT_h} +  h^{-1} \|  \bfe_{\bfw} -  \nabla e_u \|_{\calT_h} \leq Ch^s\| \bfw \|_{1+s,\Omega} . 
\end{align}
By (\ref{err_res1}) and (\ref{appro_1}), it directly follows that
\[
\| \nabla \cdot (\bfw - \bfw_h) \|_{\calT_h} \leq C h^s\| \bfw \|_{1+s,\Omega}.
\]

By Theorem~\ref{thm_discrete_embedding}, triangle inequality and trace inequality, we obtain
\begin{align}\label{err_eu_en_est1}
\Vert e_{u} \Vert_{H^{1}(\Omega)} + \Vert e_{u} \Vert_{2, \calT_{h}}
& \leq \| \nabla \cdot  (\nabla e_u)\|_{\calT_h} +  \|h^{-\frac{1}{2}}  \llbracket  \nabla e_u \cdot \bfn \rrbracket \|_{\calE_h} \\
& \leq \| \nabla \cdot  (\nabla e_u -\bfe_{\bfw} )\|_{\calT_h} + \|\nabla \cdot \bfe_{\bfw}\|_{\calT_h} +  \|h^{-\frac{1}{2}}  \llbracket ( \bfe_{\bfw}-\nabla e_u )\cdot \bfn \rrbracket \|_{\calE_h}\nn\\
&\leq C\left(  h^{-1}   \| \nabla e_u -\bfe_{\bfw} \|_{\calT_h} + \|\nabla \cdot \bfe_{\bfw}\|_{\calT_h} \right).\nn
\end{align}
By (\ref{err_eu_en_est1}), (\ref{err_res1}), (\ref{appro_3}) and triangle inequality, we obtain 
\begin{align*}
\Vert u - u_{h} \Vert_{H^{1}(\Omega)} + \Vert u - u_{h} \Vert_{2, \calT_{h}} \leq C h^s\| \bfw \|_{1+s,\Omega}.
\end{align*}
For $\theta=-1$, the term $T_0$
becomes zero and the rest of the estimates hold. This complete the rest of the proof.
\end{proof}

In order to prove the $L^2$-norm of error estimate for $u-u_h$, we apply the Aubin-Nitsche duality argument. 

Now consider 
the dual problem (\ref{dual_p})
\begin{subequations}\label{dual_p}
\begin{align}
\Delta (\kappa \Delta \varphi)= z, &  \ \ \textrm{in}  \ \ \Omega, \\
\varphi = 0, &  \ \ \textrm{on}  \ \  \partial \Omega, \\
\frac{\partial \varphi}{\partial n} = 0, &  \ \ \textrm{on}  \ \  \partial \Omega.
\end{align}
\end{subequations}
with the following elliptic regularity condition: 
$\bfpsi = \nabla \varphi \in H^{1+\alpha}(\Omega)$ with $\alpha > 1/2$ and there holds 
\begin{align}
\label{reg_dual}
\|\varphi\|_{2+\alpha} + \|\bfpsi\|_{1+\alpha,\Omega} \leq C \|z\|_{0,\Omega}. 
\end{align}

\begin{thm}
\label{thm_L2_con_biharmonic}
Let the conditions in Theorem \ref{main_err1} and the regularity result (\ref{reg_dual}) of the dual problem hold. 
Then, the following estimates hold for sufficiently large $\tau>0$,
\begin{align}
\label{l2_err_u}
\|u-u_h\|_{\calT_h} & \lesssim h^{s+\sigma} \|\bfw\|_{1+s,\Omega}, \\
\label{l2_err_w}
\|\bfw-\bfw_h\|_{\calT_h} & \lesssim  h^{\min\{1+s,s+\sigma/2\}} \|\bfw\|_{1+s,\Omega},\\
\label{l2_err_gradu}
\|\nabla u-\nabla u_h\|_{\calT_h}& \lesssim h^{\min\{1+s,s+\sigma/2\}} \|\bfw\|_{1+s,\Omega},
\end{align}
where $s = \min\{\delta,k+1\}, \sigma = \min\{ \alpha,k+1 \}$.
\end{thm}

\begin{proof}
For the dual problem (\ref{dual_p}), there holds 
\begin{align*}
B_{\theta}((\bfpsi,\varphi), (\bfw - \bfw_{h},u - u_{h}))= (z, u - u_{h}).
\end{align*} 
Since $B_{\theta}((\bfw-\bfw_h,u-u_h), (\bfeta_h,v_h)) = 0$ for 
all $(\bfeta_h,v_h) \in \bfW_h \times V_h$, we obtain
\begin{align}
B_{\theta}((\bfpsi-\bfpsi_h, \varphi-\varphi_h), (\bfw-\bfw_h,u-u_h) )  = (z, u - u_{h}),
\label{dual_est1}
\end{align}
where $(\bfpsi_h,\varphi_h)$ is the discrete solution of the dual problem (\ref{dual_p}) based on the 
mixed finite element scheme (\ref{sta_fem_biharmonic}). 
Then, we rewrite the left hand-side 
of (\ref{dual_est1}) as
\[
B_{\theta}((\bfpsi-\bfpsi_h, \varphi-\varphi_h), (\bfw-\bfw_h,u-u_h) )  = \sum^4_{k=1} D_k,
\]
where 
\begin{align*}
&D_1 = (\kappa\nabla\cdot(\bfpsi - \bfpsi_h), \nabla \cdot(\bfw-\bfw_h))_{\calT_h}, \\
& D_2 = ( \nabla(\kappa\nabla \cdot(\bfpsi-\bfpsi_h)), \bfw-\bfw_h-\nabla(u-u_h) )_{\calT_h},\\
&D_3 = \theta ( \nabla (\kappa\nabla \cdot (\bfw-\bfw_h)) ,  
\bfpsi-\bfpsi_h - \nabla(\varphi - \varphi_h))_{\calT_h}, \\
& D_4 = \tau h^{-2} ( \kappa( \bfpsi-\bfpsi_h - \nabla(\varphi - \varphi_h)), \bfw-\bfw_h-\nabla(u-u_h) )_{\calT_h}.
\end{align*}
By the error estimate in $H({\rm div})$-norm (cf. Theorem \ref{main_err1}), $D_1$ can be made bounded by
\[
D_1 \leq C h^{s + \sigma} \| \bfpsi \|_{1+\sigma,\Omega} \| \bfw \|_{1+s,\Omega}.
\]
From (\ref{err_res1}), we  easily find  that
\begin{align}
h^{-1}\| \bfw_h - \nabla u_h \|_{\calT_h} \leq C h^{s}\|\bfw\|_{1+s,\Omega}. \label{aux_err_est1}
\end{align}
Similarly, we also have the following estimate for the discrete solution of the dual problem (\ref{dual_p}) 
\[
h^{-1}\| \bfpsi_h - \nabla \varphi_h \|_{\calT_h} \leq C h^{\sigma}\|\bfpsi\|_{1+\sigma,\Omega} .
\]
Thus, a use of the Cauchy-Schwarz inequality yields
\[
D_4 \leq C h^{s+\sigma} \| \bfpsi \|_{1+\sigma,\Omega} \| \bfw \|_{1+s,\Omega}.
\]
Next we take integration by parts for $D_2$ and $D_3$. We note that $ \|   \nabla\cdot(\Pi^{\rm div}_h \bfpsi -\bfpsi_h)\|_{ \calT_h} \leq C h^{\sigma} \| \bfpsi \|_{1+\sigma,\Omega} $ can be similarly derived as in (\ref{err_res1}). Now for $D_2$, we arrive at
\begin{align*}
D_2 &=  (  \kappa\nabla\cdot(\bfpsi - \bfpsi_h), \nabla \cdot ( \bfw_h - \nabla u_h ))_{\calT_h} 
-\langle  \kappa\nabla\cdot(\bfpsi - \bfpsi_h), (\bfw_h - \nabla u_h )\cdot\bfn \rangle_{\partial \calT_h}\nn\\
&\leq C \|\nabla\cdot(\bfpsi - \bfpsi_h)  \|_{\calT_h} h^{-1}\|\bfw_h - \nabla u_h\|_{\calT_h} \\
&\quad + Ch^{\frac{1}{2}}( \| \nabla\cdot (\bfpsi - \bfPi^{\rm div}_h \bfpsi) \|_{\partial \calT_h} + \|   \nabla\cdot(\bfPi^{\rm div}_h \bfpsi -\bfpsi_h)\|_{\partial \calT_h})h^{-1} \| \bfw_h - \nabla u_h \|_{\calT_h} \\
&\leq Ch^{s+\sigma} \| \bfpsi \|_{1+\sigma,\Omega} \| \bfw \|_{1+s,\Omega}.
\end{align*}
For $D_3$, it follows that 
\begin{align*}
D_3  &= \theta (\kappa \nabla\cdot(\bfw - \bfw_h), \nabla \cdot(\bfpsi_h - \nabla \varphi_h) )_{\calT_h} 
- \theta\langle \kappa \nabla \cdot(\bfw-\bfw_h), (\bfpsi_h - \nabla \varphi_h)\cdot\bfn  \rangle_{\partial \calT_h}\\
& \leq C \theta \|\nabla\cdot(\bfw - \bfw_h)\|_{\calT_h} h^{-1} \|\bfpsi_h - \nabla \varphi_h)\|_{\calT_h}\\
&\quad + C \theta h^{\frac{1}{2}}( \| \nabla\cdot (\bfw - \bfPi^{\rm div}_h \bfw) \|_{\partial \calT_h} + \|   \nabla\cdot(\bfPi^{\rm div}_h \bfw -\bfw_h)\|_{\partial \calT_h})h^{-1} \| \bfpsi_h - \nabla \varphi_h \|_{\calT_h} \\
&\leq C \theta h^{s+\sigma} \| \bfpsi \|_{1+\sigma,\Omega} \| \bfw \|_{1+s,\Omega}.
\end{align*}
Combing the estimates for $D_l,l=1,\cdots,4$, we obtain
\[
(z, u-u_h) \leq C\;h^{s + \sigma} \| \bfpsi \|_{1+\sigma,\Omega} \| \bfw \|_{1+s,\Omega}.
\]
Let $z = u-u_h$. Then, we obtain the desired estimate (\ref{l2_err_u}) by (\ref{reg_dual}).

Next, we decompose $\|\bfe_{\bfw}\|^2_{\calT_h}$ into 
\[
\|\bfe_{\bfw}\|^2_{\calT_h} = ( \bfPi^{\rm div}_h \bfw - \bfw  + \nabla u - \nabla u_h 
+\nabla u_h - \bfw_h, \bfe_{\bfw} )_{\calT_h}.
\]
By (\ref{appro_2}) and (\ref{aux_err_est1}), we arrive at
\begin{align}\label{errw_pre1}
\|\bfe_{\bfw}\|^2_{\calT_h} \leq Ch^{1+s} \|\bfw\|_{1+s,\Omega} \|\bfe_{\bfw}\|_{\calT_h}  
+ (\nabla u - \nabla u_h,\bfe_{\bfw}  )_{\calT_h}.
\end{align}
By integration by parts and noting that $\bfe_{\bfw}  \in \bfW_h$ and $u-u_h \in H^1_0(\Omega)$, it follows that 
\begin{align}
(\nabla u - \nabla u_h,\bfe_{\bfw}  )_{\calT_h} &= -( u -  u_h, \nabla \cdot \bfe_{\bfw}  )_{\calT_h} \label{errw_pre2} \\
&\leq C \| u-u_h \|_{\calT_h} \|\nabla \cdot \bfe_{\bfw}  \|_{ \calT_h} \nn \\
&\leq  C h^{2s + \sigma} \|\bfw\|^2_{1+s,\Omega} .\nn
\end{align}
Now (\ref{l2_err_w}) can be obtained by (\ref{errw_pre1}), (\ref{errw_pre2}) and (\ref{appro_2}).

Similarly, we can decompose $\|\nabla e_u\|^2_{\calT_h}$ into 
\[
\|\nabla e_u\|^2_{\calT_h} = ( \bfw-\bfw_h + \bfw_h - \nabla u_h, \nabla e_u)_{\calT_h}.
\]
Then, the estimate (\ref{l2_err_gradu}) follows from the above equality, (\ref{l2_err_w}) and (\ref{aux_err_est1}). 
This completes the rest of the proof.
\end{proof}

\begin{remark}
\label{remark_conv}
When $\kappa$ is a constant and the domain is convex, the solution of (\ref{equ_biharmonic}) are smooth 
enough and the solution of dual problem (\ref{dual_p}) is also smooth with $\alpha=2$, then we have
\begin{align*}
 \|u-u_h\|_{\calT_h} &\leq C h^{k+3}\|\bfw\|_{k+2,\Omega},\quad k\geq 1,\\
 \|\bfw-\bfw_h\|_{\calT_h} +\|\nabla u-\nabla u_h\|_{\calT_h}& \leq C h^{k+2}\|\bfw\|_{k+2,\Omega},\quad k\geq 1,\\
 \| \nabla \cdot (\bfw - \bfw_h) \|_{\calT_h} + \| \nabla (u -  u_h) \|_{1,\calT_h} &\leq C h^{k+1}\| 
 \bfw \|_{k+2,\Omega},\quad k\geq 0.
\end{align*}
\end{remark}

\section{A new mixed finite element scheme for the von K\'arm\'an equation}\label{sec_method_vonkarman}
In this section, we extend our new formulation to 
the von K\'arm\'an equation. Further, under smallness condition on the data to be defined subsequently, 
we present the existence 
and uniqueness result for the discrete
nonlinear system, the {\it a priori} bounds 
and the corresponding error estimates.

Through out this section, we assume that the following regularity of the solution of 
the von K\'arm\'an equation (\ref{equ_von_karman}):
 \begin{align}\label{ass_reg_vk}
 \xi \in  H^{2+\beta}(\Omega),  \psi \in H^{2+\beta}(\Omega), 
 \quad  \text{ where } 2 \geq\beta > 1/2.
\end{align}
We first recall the result of in \cite[Lemma $2.2$]{Brenner2017}. For the solution $ (\xi,\psi)$ of the von K\'arm\'an equation (\ref{equ_von_karman}) and any $\eta \in H^1_0(\Omega)$, there holds
\begin{align}\label{lem_equ_pre_vk}
([\xi,\psi],\eta) = ({\rm cof}(D^2\xi): D^2 \psi, \eta) 
=- ({\rm cof}(D^2\xi)\nabla \psi,\nabla \eta).
\end{align}

We use the same finite element spaces $\bfW_h$ and $V_h$ as in Section \ref{sec_method_biharmonic} in 
two dimensions. With $\bfu=\nabla \xi$ and $\bfw=\nabla \psi$, our mixed finite element scheme 
for the von K\'arm\'an equation (\ref{equ_von_karman}) is to seek an approximation 
$(\bfu_h,\bfw_h,\xi_h,\psi_h)\in \bfW_h\times \bfW_h\times V_h \times V_h$ such that
\begin{subequations}\label{sta_method_vk}
\begin{align}
B_{\theta}((\bfu_h,\xi_h),(\bfv,\eta))+({\rm cof}(D^2\xi_h)\nabla \psi_h,\nabla \eta)_{\calT_h} &
= (f,\eta)_{\langle H^{-1}(\Omega), H_{0}^{1}(\Omega) \rangle},\\
B_{\theta} ((\bfw_h,\psi_h),(\bfz,\phi)) -({\rm cof}(D^2\xi_h)\nabla \xi_h,\nabla \phi)_{\calT_h} &= 0,
\end{align}
\end{subequations}
for any $(\bfv,\bfz,\eta,\phi) \in \bfW_h\times \bfW_h\times V_h \times V_h$, where the bilinear form 
$B_{\theta}((\cdot,\cdot),(\cdot,\cdot))$ is defined as in (\ref{def_bilinear}) with $\kappa = 1$.

\subsection{Existence and uniqueness of the discrete nonlinear system and stability estimate} 
Since the discrete system  (\ref{sta_method_vk}) leads to a system of nonlinear algebraic system,
therefore, in this subsection, we first prove the existence and uniqueness of the discrete system based on  
the following one point iterative scheme, called the Picard's method.
 from the mixed finite element scheme (\ref{sta_method_vk}). 

Given an initialization $\xi^{m-1}_h \in V_h$, $m\geq 1$, the Picard's iteration for the nonlinear system (\ref{sta_method_vk}) 
is to find $(\bfu^m_h,\bfw^m_h,\xi^m_h,\psi^m_h)\in \bfW_h\times \bfW_h\times V_h \times V_h$ such that
\begin{subequations}\label{picard_iter}
\begin{align}
B_{\theta}((\bfu^m_h,\xi^m_h),(\bfv,\eta))+\left({\rm cof}(D^2\xi^{m-1}_h)\nabla \psi^m_h,\nabla \eta\right)_{\calT_h} 
&= (f,\eta)_{\langle H^{-1}(\Omega), H_{0}^{1}(\Omega) \rangle},\\
B_{\theta}((\bfw^m_h,\psi^m_h),(\bfz,\phi)) - \left({\rm cof}(D^2\xi^{m-1}_h)\nabla \xi^m_h,\nabla \phi\right)_{\calT_h} &= 0,
\end{align}
\end{subequations}
for any $(\bfv,\bfz,\eta,\phi) \in \bfW_h\times \bfW_h\times V_h \times V_h$. Choosing $\bfv = \bfu^m_h, 
\eta = \xi^m_h, \bfz = \bfw^m_h, \phi =\psi^m_h $ and noting that
\[
({\rm cof}(D^2\xi^{m-1}_h)\nabla \psi^m_h,\nabla \xi^m_h)_{\calT_h}  
= ({\rm cof}(D^2\xi^{m-1}_h)\nabla \xi^m_h,\nabla \psi^m_h)_{\calT_h} ,
\]
we now obtain
\[
B_{\theta}((\bfu^m_h,\xi^m_h),(\bfu^m_h,\xi^m_h)) + B((\bfw^m_h,\psi^m_h),(\bfw^m_h,\psi^m_h)) 
=  (f,\xi^m_h)_{\langle H^{-1}(\Omega), H_{0}^{1}(\Omega) \rangle}.
\]
Similar to the stability estimate (\ref{main_sta_est}), 
we easily follow the proof of Theorem \ref{thm_main_sta} to derive the  estimate below, provided  
for $\theta\neq -1$ the stabilization parameter $\tau$ in (\ref{sta_method_vk}) is chosen to be large enough,
and for  $\theta=-1$,  $\tau$ is chosen to be any arbitrary positive constant.
\begin{align*}
\||(\bu_h^m,\bfw_h^m, \xi_h^m, \psi_h^m)\|| 
 \leq C_{sta} \|f\|_{H^{-1}(\Omega)},
\end{align*}
where
$$\||(\bu_h^m,\bfw_h^m, \xi_h^m, \psi_h^m)\||:= \Big(\|\bfu^m_h\|^2_{H({\rm div},\Omega)} 
+ \|\bfw^m_h\|^2_{H({\rm div},\Omega)} + \| \xi^m_h\|^2_{2,\calT_h}  
 + \| \psi^m_h\|^2_{2,\calT_h}\Big)^{1/2}.$$
Inspired by the above result, we define a closed subset of $\bfW_h\times \bfW_h\times V_h \times V_h$:
\begin{align}
\label{def_K_h}
\calK_h:= \{ (\bfv, \bfz, \eta, \phi) \in  \bfW_h\times \bfW_h\times V_h \times V_h: \ 
\||(\bfv,\bfz, \eta, \phi)\|| \leq C_{sta} \|f\|_{H^{-1}(\Omega)} \}.
\end{align}
We also define a mapping $\calF: \calK_h \rightarrow \calK_h$ as follows: for any $(\widehat{\bfu},\widehat{\bfw},\widehat{\xi},\widehat{\psi})\in \calK_h$, $(\bfu^\ast,\bfw^\ast,\xi^\ast,\psi^\ast)= \calF(\widehat{\bfu},\widehat{\bfw},\widehat{\xi},\widehat{\psi})$ is obtained by one step of the above Picard iteration.  Clearly, $(\bfu_h,\bfw_h,\xi_h,\psi_h)$ is a solution of (\ref{sta_method_vk}) if and only if it is a fixed point of the mapping $\calF$.

We are now ready to show the existence and uniqueness result for the nonlinear system (\ref{sta_method_vk}) 
and the associated stability estimate.

\begin{thm}
\label{thm_sta_vk}
(Existence, uniqueness and stability) 
If $\|f\|_{H^{-1}(\Omega)}$ is small enough and the stabilization 
parameter $\tau$ in (\ref{sta_method_vk}) is large enough for $\theta = 1$ and for $\theta=-1$ any $\tau>0$, 
the discrete nonlinear system (\ref{sta_method_vk}) has
a unique solution $(\bfu_h,\bfw_h,\xi_h,\psi_h)\in \bfW_h\times \bfW_h\times V_h \times V_h$. More over, 
there also holds the following  estimate::
\begin{align}
\label{main_thm_sta_vk}
\||(\bfu_h,\bfw_h, \xi_h, \psi_h)\||
  \leq C_{sta} \|f\|_{H^{-1}(\Omega)}. 
\end{align}
\end{thm}
\begin{proof}
Clearly, the mapping $\calF$ maps $\calK_h$ into itself. In order to show the existence and uniqueness of 
the solution of (\ref{sta_method_vk}), it suffices to show that $\calF$ is a contraction on $\calK_h$. Let 
$
(\widehat{\bfu}^1_h,\widehat{\bfw}^1_h,\widehat{\xi}^1_h,\widehat{\psi}^1_h)$, 
$(\widehat{\bfu}^2_h,\widehat{\bfw}^2_h,\widehat{\xi}^2_h,\widehat{\psi}^2_h) \in \calK_h
$
and $({\bfu}^1_h,{\bfw}^1_h,{\xi}^1_h,{\psi}^1_h), ({\bfu}^2_h,{\bfw}^2_h,{\xi}^2_h,{\psi}^2_h)$ be the solutions 
of the Picard iteration (\ref{picard_iter}) with the initializations 
$(\widehat{\bfu}^1_h,\widehat{\bfw}^1_h,\widehat{\xi}^1_h,\widehat{\psi}^1_h), (\widehat{\bfu}^2_h,
\widehat{\bfw}^2_h,\widehat{\xi}^2_h,\widehat{\psi}^2_h),$ respectively. Define
\[
\delta_\bfu := \bfu^1_h - \bfu^2_h,\quad \delta_\bfw := \bfw^1_h - \bfw^2_h, \quad \delta_\xi := \xi^1_h - \xi^2_h,\quad \delta_\psi := \psi^1_h - \psi^2_h.
\]
By (\ref{picard_iter}), it follows that
\begin{align*}
B_{\theta}((\delta_\bfu,\delta_\xi),(\bfv,\eta))+\left({\rm cof}(D^2\widehat{\xi}^1_h)\nabla \psi^1_h 
- {\rm cof}(D^2\widehat{\xi}^2_h)\nabla \psi^2_h,\nabla \eta\right)_{\calT_h} &= 0,\\
B_{\theta}((\delta_\bfw,\delta_\psi),(\bfz,\phi)) -\left({\rm cof}(D^2\widehat{\xi}^1_h)\nabla \xi^1_h 
-{\rm cof}(D^2\widehat{\xi}^2_h)\nabla \xi^2_h  ,\nabla \phi\right)_{\calT_h} &= 0.
\end{align*}
Choose $\bfv = \delta_\bfu, \bfz = \delta_\bfw, \eta = \delta_\xi, \phi =\delta_\psi $. Following the proof of 
stability estimate for (\ref{main_sta_est}) again, if the stabilization parameter $\tau>0$ 
in (\ref{sta_method_vk}) is large enough for $\theta\neq -1$ and for $\theta=-1$ any arbitrary $\tau>0$, 
we easily obtain
\begin{align}
\label{vk_exi_est1}
\||(\delta_\bfu,\delta_\bfw,\delta_\xi,\delta_\psi)\||^2 \leq C^2_{sta} \left|T_\delta\right|, 
\end{align} 
where $T_\delta =   - \left({\rm cof}(D^2\widehat{\xi}^1_h)\nabla \psi^1_h - {\rm cof}(D^2\widehat{\xi}^2_h)\nabla \psi^2_h,\nabla \delta_\xi \right)_{\calT_h}  +  \left({\rm cof}(D^2\widehat{\xi}^1_h)\nabla \xi^1_h -{\rm cof}(D^2\widehat{\xi}^2_h)\nabla \xi^2_h  ,\nabla \delta_\psi\right)_{\calT_h}.$

In order to get an upper bound of $T_\delta$, we rewrite $T_\delta$ as follows:
\begin{align*}
T_\delta  = R_1 + R_2, 
\end{align*}
where
\begin{align*}
R_1 = & - \left({\rm cof}(D^2\widehat{\xi}^1_h-D^2\widehat{\xi}^2_h)\nabla \psi^1_h,\nabla \delta_\xi \right)_{\calT_h}
 - \left({\rm cof}(D^2\widehat{\xi}^2_h) \nabla \delta_\psi, \nabla \delta_\xi \right)_{\calT_h}, \\  
R_2 = & \left({\rm cof}(D^2\widehat{\xi}^1_h-D^2\widehat{\xi}^2_h)\nabla \xi^1_h  ,\nabla \delta_\psi\right)_{\calT_h}
+\left({\rm cof}(D^2\widehat{\xi}^2_h) \nabla \delta_\xi ,\nabla \delta_\psi\right)_{\calT_h} .
\end{align*}

For $R_1$, we have
\begin{align} 
\label{vk_exi_est2}
|R_1| &\leq \|D^2\widehat{\xi}^1_h-D^2\widehat{\xi}^2_h\|_{\calT_{h}}  \|\nabla \psi^1_h\|_{L^4(\Omega)} 
\|\nabla \delta_\xi \|_{L^4(\Omega)} + \|D^2\widehat{\xi}^2_h\|_{\calT_{h}} 
\Vert \nabla \delta_\psi\Vert_{L^{4}(\Omega)} \Vert \nabla \delta_\xi \Vert_{L^{4}(\Omega)} \\
&\leq C_1 \left( \|\widehat{\xi}^1_h-\widehat{\xi}^2_h\|_{2, \calT_{h}} 
\| \psi^1_h\|_{2,\calT_h} \| \delta_\xi\|_{2,\calT_h} 
+ \|\widehat{\xi}^2_h\|_{2, \calT_{h}} 
\| \delta_{\psi}\|_{2,\calT_h} \| \delta_\xi\|_{2,\calT_h} \right)
\quad \text{(by \cite[Theorem $2.1$]{Ern2010a})} \nn\\\
& \leq C_1 C_{sta} \;\|f\|_{H^{-1}(\Omega)}
\left(\|\widehat{\xi}^1_h-\widehat{\xi}^2_h\|_{2,\calT_h} + \| \delta_{\psi}\|_{2,\calT_h} \right) 
\| \delta_\xi\|_{2,\calT_h} \qquad\quad \ \text{(by the property of  $\calK_h$)}\nn\\
& \leq \frac{1}{2} C_1 C_{sta} \|f\|_{H^{-1}(\Omega)} \left( \|\widehat{\xi}^1_h-\widehat{\xi}^2_h\|^2_{2,\calT_h} 
+ \| \delta_{\psi}\|_{2,\calT_h}^{2}+ 2 \| \delta_\xi\|^2_{2,\calT_h}\right). \nn
\end{align}
Then by the property of the subspace $\calK_h$, $R_2$ can be similarly deduced as follows:
\begin{align} \label{vk_exi_est3}
|R_2| & \leq\frac{1}{2} C_1 C_{sta}\;\|f\|_{H^{-1}(\Omega)} \left(\|\widehat{\xi}^1_h-\widehat{\xi}^2_h\|^2_{2,\calT_h} 
+ \| \delta_\xi\|^2_{2,\calT_h} + 2 \|\delta_\psi\|^2_{2,\calT_h}\right).
\end{align}
With $\delta_{\widehat{\xi}} = \widehat{\xi}^1_h-\widehat{\xi}^2_h$, combine (\ref{vk_exi_est2})  and 
(\ref{vk_exi_est3}) to obtain
\begin{align} \label{vk_exi_est3-1}
|T_{\delta}| &\leq |R_1|+ |R_2| 
 \leq C_1 C_{sta}\;\|f\|_{H^{-1}(\Omega)} \left(\|\delta_{\widehat{\xi}}\|^2_{2,\calT_h} 
+ \frac{3}{2}  \| \delta_\xi\|^2_{2,\calT_h} + \frac{3}{2} \|\delta_\psi\|^2_{2,\calT_h}\right)\nn\\
& \leq C_1 C_{sta}\;\|f\|_{H^{-1}(\Omega)} \Big(\||(\delta_{\widehat{\bfu}},\delta_{\widehat{\bfw}},
\delta_{\widehat{\xi}},\delta_{\widehat{\psi}})\||^2 + \frac{3}{2} \||(\delta_\bfu,\delta_\bfw,\delta_\xi,\delta_\psi)\||^2 
\Big).
\end{align}
On sustitution of (\ref{vk_exi_est3-1}) in (\ref{vk_exi_est1}), we  find that
\begin{align*}
(1- \frac{3}{2} C_1 C_{sta}^3 \;\|f\|_{H^{-1}(\Omega)})\;  \||(\delta_\bfu,\delta_\bfw,\delta_\xi,\delta_\psi)\||^2
\leq C_1 C_{sta}^3\;\|f\|_{H^{-1}(\Omega)} \||(\delta_{\widehat{\bfu}},\delta_{\widehat{\bfw}},
\delta_{\widehat{\xi}},\delta_{\widehat{\psi}})\||^2.
\end{align*} 
Choose $\|f\|_{H^{-1}\Omega} \leq 2/(3C_1\;C_{sta}^3)$ and obtain
\begin{align*}
 \||(\delta_\bfu,\delta_\bfw,\delta_\xi,\delta_\psi)\||^2
\leq \frac{C_1 C_{sta}^3\;\|f\|_{H^{-1}(\Omega)}}{(1- \frac{3}{2} C_1 C_{sta}^3 \;\|f\|_{H^{-1}(\Omega)})}\; 
\||(\delta_{\widehat{\bfu}},\delta_{\widehat{\bfw}},
\delta_{\widehat{\xi}},\delta_{\widehat{\psi}})\||^2.
\end{align*} 
Obviously, the above bound implies that $\calF$ is a contraction on $\calK_h$ if we further 
choose $\|f\|_{H^{-1}(\Omega)}<\frac{1}{3C_1\;C_{sta}^3}$ such that 
$\lambda: =  \frac{C_1 C_{sta}^3 \|f\|_{H^{-1}(\Omega)}}{1- \frac{3}{2} C_1 C_{sta}^3 \|f\|_{H^{-1}(\Omega)})}< 1$. 

By the fixed point theorem, there is a unique fixed point $(\bfu_h,\bfw_h,\xi_h,\psi_h)$ of the mapping $\calF$, and it is also the unique solution of the system (\ref{sta_method_vk}). 

Now we have proved the existence and uniqueness of (\ref{sta_method_vk}).
Notice that (\ref{main_thm_sta_vk}) holds due to the definition of $\calK_{h}$ in (\ref{def_K_h}).
Now we complete the proof.
\end{proof}

\subsection{Error estimates}
In this section, we present  {\it a priori} error estimates for the mixed finite element 
scheme (\ref{sta_method_vk}) for the von K\'arm\'an equation. 

Define
\[
\bfe_{\bfu} :=  \bfPi^{\rm div}_h\bfu - \bfu_h,\quad \bfe_{\bfw} :=  \bfPi^{\rm div}_h\bfw - \bfw_h, \quad e_{\xi} :=  \pi_h \xi -\xi_h,\quad e_\psi = \pi_h \psi - \psi_h,
\]
where $\bfPi^{\rm div}_h: H_{0}({\rm div},\Omega) \rightarrow \bfW_h$ and $\pi_h: H^1_0(\Omega)\rightarrow V_h$ 
are the $H({\rm div})$-smooth  projection and $H^1$-smooth projection as used in Section \ref{sec_err_b}. 
By the regularity assumption (\ref{ass_reg_vk}), the following approximation properties for the solution 
$(\bfu,\xi)$ of (\ref{equ_von_karman}) with $\bfu = \nabla \xi$ hold true for the two operators $\bfPi^{\rm div}_h$ and $\pi_h$:
\begin{subequations}
\label{appros_vk}
\begin{align}
\label{appro_1_vk}
\| \bfu - \bfPi^{\rm div}_h\bfu\|_{H({\rm div},\Omega)}   + h^{\frac{1}{2}}\| \nabla \cdot( \bfu - \bfPi^{\rm div}_h\bfu  )\|_{\partial \calT_h}  & \leq C h^{r} \| \bfu \|_{1+r,\Omega},\\
 \label{appro_2_vk}
\| \bfu - \bfPi^{\rm div}_h\bfu\|_{\calT_h} + \| \nabla ( \xi - \pi_h \xi) \|_{\calT_h} & \leq Ch^{1+r} \|\bfu \|_{1+r,\Omega},\\
\label{appro_3_vk}
\Vert \xi - \pi_{h} \xi \Vert_{2, \calT_{h}}
& \leq C h^{r} \| \bfu \|_{1+r,\Omega},
\end{align}
\end{subequations}
where $r =\min\{\beta, k+1\}$. We remark that (\ref{appro_1_vk}) and (\ref{appro_2_vk}) also hold true for the solution $(\bfw,\psi)$ of (\ref{equ_von_karman}) with $\bfw = \nabla \psi$.

We first present the error equation which we need for the error estimates.
\begin{lem}\label{lemma_erq_vk}
With $\bfu = \nabla \xi$ and $\bfw = \nabla \psi,$ let $(\xi,\psi)$ and $(\bfu_h,\bfw_h,\xi_h,\psi_h)$ be 
the solutions of (\ref{equ_von_karman}) 
and (\ref{sta_method_vk}),respectively. Under the regularity assumption  (\ref{ass_reg_vk}),  there holds 
for any $(\bfv,\bfz,\eta,\phi) \in \bfW_h\times \bfW_h\times V_h \times V_h$:
\begin{align}\label{error_eq_vk}
B_{\theta}((\bfe_{\bfu},e_{\xi}),(\bfv,\eta)) + B_{\theta}((\bfe_{\bfw},e_{\psi}),(\bfz,\phi))  = S_1 + S_2 + S_3 + S_4,
\end{align}
where 
\begin{align*}
S_1 &=  -( \nabla \cdot( \bfu -  \bfPi^{\rm div}_h \bfu ),\nabla \cdot \bfv )_{\calT_h}  - \tau h^{-2} ( \bfu -  \bfPi^{\rm div}_h \bfu 
 - \nabla( \xi-\pi_h \xi ), \bfv- \nabla \eta  )_{\calT_h} \nn\\
& \quad +( \nabla \cdot(\bfu -\bfPi^{\rm div}_h \bfu ) , \nabla \cdot(\bfv - \nabla \eta) )_{\calT_h} 
- \langle \nabla \cdot (\bfu -\bfPi^{\rm div}_h \bfu), (\bfv - \nabla \eta) \cdot \bfn \rangle_{\partial \calT_h}\nn\\
&\quad -\theta ( \bfu -  \bfPi^{\rm div}_h \bfu -  \nabla( \xi-\pi_h \xi ), \nabla (\nabla \cdot \bfv))_{\calT_h},
\end{align*}
\begin{align*}
S_2 & = -( \nabla \cdot( \bfw -  \bfPi^{\rm div}_h \bfw ),\nabla \cdot \bfz )_{\calT_h}  - \tau h^{-2} ( \bfw -  \bfPi^{\rm div}_h \bfw  - \nabla( \psi-\pi_h \psi ), \bfz - \nabla \phi  )_{\calT_h} \nn\\
& \quad +( \nabla \cdot(\bfw -\bfPi^{\rm div}_h \bfw ) , \nabla \cdot(\bfz - \nabla \phi) )_{\calT_h} - \langle \nabla \cdot (\bfw -\bfPi^{\rm div}_h \bfw), (\bfz - \nabla \phi) \cdot \bfn \rangle_{\partial \calT_h}\nn\\
&\quad - \theta ( \bfw -  \bfPi^{\rm div}_h \bfw -  \nabla( \psi-\pi_h \psi ), \nabla (\nabla \cdot \bfz))_{\calT_h},
\end{align*}
\begin{align*}
S_3 =  - \left( {\rm cof}(D^2 \xi) \nabla \psi- {\rm cof}(D^2 \xi_h) \nabla \psi_h,\nabla \eta \right)_{\calT_h},
\end{align*}
\begin{align*}
S_4 =  \left( {\rm cof}(D^2 \xi) \nabla \xi -{\rm cof}(D^2 \xi_h) \nabla \xi_h,\nabla \phi  \right)_{\calT_h}.
\end{align*}
\end{lem}
\begin{proof}
With (\ref{lem_equ_pre_vk}), the desired error equation is obtained by applying similar derivation in 
Lemma \ref{lem_err_bih} to (\ref{sta_method_vk}).
\end{proof}

Below, we discuss the main result of this section.

\begin{thm}
\label{thm_conv_von_karman}
Further under the conditions in Lemma \ref{lemma_erq_vk} and smallness condition on $\|f\|_{H^{-1}(\Omega)},$ 
 if the stabilization parameter $\tau$ 
in (\ref{sta_method_vk}) is large enough when $\theta = 1$ and is arbitrary positive $\tau$ for $\theta=-1,$ 
then following estimate holds:
\begin{align*}
\Vert \xi - \xi_{h} \Vert_{2, \calT_{h}} + \Vert \psi - \psi_{h} \Vert_{2, \calT_{h}}
+\| (\bfu - \bfu_h) \|_{H({\rm div},\Omega)} + \|(\bfw - \bfw_h) \|_{H({\rm div},\Omega)} 
 \leq C h^r\left(\| \bfu \|_{1+r,\Omega} + \| \bfw \|_{1+r,\Omega} \right),
\end{align*}
where $r= \min\{\beta, k+1\}$ and $\beta$ is the parameter given in (\ref{ass_reg_vk}).
\end{thm}
\begin{proof}
Taking $\bfv =\bfe_{\bfu} ,\bfz =\bfe_{\bfw} ,\eta =e_{\xi} ,\phi = e_{\psi}$ in the error equation (\ref{error_eq_vk}) 
and applying the similar technique to deal with the terms $B((\bfe_{\bfu},e_{\xi}),(\bfe_{\bfu},e_{\xi})), 
B((\bfe_{\bfw},e_{\psi}),(\bfe_{\bfw},e_{\psi}))$ 
and $S_1, S_2$, we can obtain
\begin{align} \label{eq:error-1}
\| e_{\xi}\|^2_{2,\calT_h} +  \| e_{\psi}\|^2_{2,\calT_h}
+\| \nabla \cdot \bfe_{\bfu} \|^2_{\calT_h} + \| \nabla \cdot  \bfe_{\bfw}\|^2_{\calT_h}  
\leq Ch^{2r} (\|\bfu\|^2_{1+r,\Omega} + \|\bfw\|^2_{1+r,\Omega} ) +C (|S_3| + |S_4|),
\end{align}
where $S_3$ and $S_4$ are now written as follows:
\begin{align*}
S_3 &= - \left( {\rm cof}(D^2 \xi) \nabla \psi- {\rm cof}(D^2 \xi_h) \nabla \psi_h,\nabla e_{\xi} \right)_{\calT_h},\\
S_4 &=  \left( {\rm cof}(D^2 \xi) \nabla \xi -{\rm cof}(D^2 \xi_h) \nabla \xi_h,\nabla e_{\psi}  \right)_{\calT_h}.
\end{align*}
For $S_3$, we obtain
\begin{align*}
S_3 &= -\left( ({\rm cof}(D^2 \xi) - {\rm cof}(D^2 \xi_h) ) \nabla \psi,\nabla e_{\xi} \right)_{\calT_h}
 - \left( {\rm cof}(D^2 \xi_h) ( \nabla \psi - \nabla\psi_h),\nabla e_{\xi} \right)_{\calT_h}\\
&\leq C \|\nabla \psi\|_{L^{4}(\Omega)}\|{\rm cof}(D^2 \xi)
 - {\rm cof}(D^2 \xi_h) \|_{\calT_h}\|\nabla e_{\xi} \|_{L^{4}(\calT_h)} \\
&\quad + C \|{\rm cof}(D^2 \xi_h)\|_{\calT_h}\|\nabla \psi - \nabla\psi_h\|_{L^4(\calT_h)}
\|\nabla e_{\xi} \|_{L^{4}(\calT_h)}\\
& \leq C  \|\nabla \psi\|_{L^{4}(\Omega)} \| \xi-\xi_h\|_{2,\calT_h} \| e_{\xi}\|_{2,\calT_h}\\
&\quad   + C \|{\rm cof}(D^2 \xi_h)\|_{\calT_h}\| \psi - \psi_h\|_{2,\calT_h}\| e_{\xi} \|_{2,\calT_h}
 \quad \text{(by \cite[Theorem $2.1$]{Ern2010a})}\\
& \leq C  \|\nabla \psi\|_{L^{4}(\Omega)}(  \| \xi-\pi_h \xi\|_{2,\calT_h} + \| e_{\xi}\|_{2,\calT_h}   )\| e_{\xi}\|_{2,\calT_h}\\
&\quad + C  \|{\rm cof}(D^2 \xi_h)\|_{\calT_h}(  \| \psi - \pi_h \psi\|_{2,\calT_h}  + \| e_{\psi} \|_{2,\calT_h})
\| e_{\xi} \|_{2,\calT_h} \\
& \leq C  \|\psi\|^2_{H^{2}(\Omega)}\Big( h^r \|\xi\|_{2+r} \;\| e_{\xi}\|_{2,\calT_h} 
+ \| e_{\xi}\|^2_{2,\calT_h} \Big)\\
&\quad +C \|{\rm cof}(D^2 \xi_h)\|_{\calT_h} \Big(  
 h^r \|\psi \|_{2+r,\Omega} \| e_{\xi}\|_{2,\calT_h}
  +  \| e_{\psi} \|_{2,\calT_h}\| e_{\xi} \|_{2,\calT_h}\Big).
\end{align*}
Similarly, for $S_4$, we establish
\begin{align*}
S_4  &= \left( ({\rm cof}(D^2 \xi) -{\rm cof}(D^2 \xi_h) ) \nabla \xi ,\nabla e_{\psi}  \right)_{\calT_h} + \left( {\rm cof}(D^2 \xi_h) (\nabla \xi  -  \nabla  \xi_h),\nabla e_{\psi}  \right)_{\calT_h}\\
&\leq C \|\nabla \xi\|_{L^{4}(\Omega)}\|{\rm cof}(D^2 \xi) -{\rm cof}(D^2 \xi_h)\|_{\calT_h} 
\|\nabla e_{\psi}  \|_{L^{4}(\calT_h)} \\
&\quad +C \|{\rm cof}(D^2 \xi_h)\|_{\calT_h} \|\nabla \xi  -  \nabla  \xi_h\|_{L^4(\calT_h)}\|\nabla e_{\psi}  \|_{L^4(\calT_h)}\\
& \leq C  \|\nabla \xi\|_{L^{4}(\Omega)} \| \xi-\xi_h\|_{2,\calT_h}  \| e_{\psi}  \|_{2,\calT_h}\\
&\quad + C\|{\rm cof}(D^2 \xi_h)\|_{\calT_h}  \| \xi  -  \xi_h\|_{2,\calT_h}\| e_{\psi}  \|_{2,\calT_h}
\quad \text{(by \cite[Theorem $2.1$]{Ern2010a})}\\
&\leq C (\|\nabla \xi\|_{L^{4}(\Omega)} +\|{\rm cof}(D^2 \xi_h)\|_{\calT_h}   )( \| \xi-\pi_h\xi \|_{2,\calT_h}
 + \| e_{\xi}\|_{2,\calT_h}  ) \| e_{\psi}  \|_{2,\calT_h}\\ 
&\leq C(\|\nabla \xi\|_{L^{4}(\Omega)} +\|{\rm cof}(D^2 \xi_h)\|_{\calT_h}   )(h^r\|\nabla \xi\|_{1+r,\Omega}
 + \| e_{\xi}\|_{2,\calT_h} ) \| e_{\psi}  \|_{2,\calT_h}.
\end{align*}
On substitution in (\ref{eq:error-1}), and using regularity property (\ref{ass_reg_vk}) with stability result
(\ref{main_thm_sta_vk}), we arrive at 
\begin{align} \label{eq:error-2}
(1-C \|f\|_{H^{-1}(\Omega)}) \Big(\| e_{\xi}\|^2_{2,\calT_h} +  \| e_{\psi}\|^2_{2,\calT_h}\Big)
+\| \nabla \cdot \bfe_{\bfu} \|^2_{\calT_h} + \| \nabla \cdot  \bfe_{\bfw}\|^2_{\calT_h}  
\leq Ch^{2r} \;\Big(\|\bfu\|^2_{1+r,\Omega} + \|\bfw\|^2_{1+r,\Omega}\Big).
\end{align}
Choose $\|f\|_{H^{-1}(\Omega)}$ is small enough so that $(1-C \|f\|_{H^{-1}(\Omega)})>0$  and
this completes the rest of the proof.
\end{proof}

\section{Numerical experiments}\label{sec:numer_example}
In this section, we present numerical results that illustrate the efficiency and accuracy of the newly proposed mixed finite element schemes for the biharmonic equation and von K\'arm\'an model respectively. In the following we focus on the numerical experiments in two dimensional cases. The grids we use to test are always assumed to be shape regular and quasi-uniform. In the descretization, we always use the BDM finite element space as vector function space and continuous polynomial space as scalar function space. For brevity, the mixed finite element space used in the following computation is denoted by ${BDM}_k$-${P}_{k+1}$ with $k\geq 1$. The computation is performed with FEniCS \cite{FEniCS}. We remark that our algorithm can be straightforwardly extended to the problem with non-homogeneous boundary conditions and other types of von K\'arm\'an model such as shown in \cite{Brenner2017,Carstensen19}.

Firstly, we presents some numerical results for the biharmonic equation.

\begin{example}\label{example1}
We first test our algorithm for the biharmonic equation (\ref{equ_biharmonic}) with $\kappa=1$ on a unit square domain $[0,1]^2$ with exact solution 
\[
u = x^2 (1-x)^2 y^2 (1-y)^2.
\]
The source term $f$ can be determined by the above exact solution. For the objective of flexibility in the design of our algorithm, we use an optional parameter $\theta$ in the variational formulation (\ref{sta_fem_biharmonic}) for the biharmonic equation.

\begin{table}[htbp]\renewcommand{\arraystretch}{1.2}
\footnotesize
\begin{tabular}{c|c c c c c}
\hline 
Errors &$\|e_u\|_{\Ct_h}$ & $\| \nabla e_u\|_{\Ct_h} $&$ \|e_{\bfw}\|_{\Ct_h} $&$\|\nabla \cdot e_{\bfw}\|_{\Ct_h} $ &$\|\nabla e_u\|_{1,\Ct_h} $\tabularnewline
\hline  
$\theta = -1,1$&8.89e-7& 4.24e-6  & 3.15e-6  & 9.11e-4 & 1.15e-3 \tabularnewline
\hline 
\end{tabular}
\smallskip
\smallskip
\caption{\footnotesize (Example \ref{example1}) Errors with respect to different choices of $\theta = -1,1$ and 
fixed $\tau =10$ on a fixed mesh with mesh size $h=0.011$ based on ${BDM}_1$-${P}_{2}$.}\label{table1}
\end{table}

\begin{table}[htbp]\renewcommand{\arraystretch}{1.2}
\footnotesize
\begin{tabular}{c|c c c c c}
\hline 
Errors &$\|e_u\|_{\Ct_h}$ & $\| \nabla e_u\|_{\Ct_h} $&$ \|e_{\bfw}\|_{\Ct_h} $&$\|\nabla \cdot e_{\bfw}\|_{\Ct_h} $ &$\|\nabla e_u\|_{1,\Ct_h} $\tabularnewline
\hline  
$\tau = 1$ & 1.21e-5 & 5.41e-5  & 2.34e-5 & 9.10e-4 & 1.19e-3 \tabularnewline
\hline 
$\tau = 10$& 8.89e-7 & 4.24e-6 & 3.15e-6  & 9.11e-4 & 1.15e-3 \tabularnewline
\hline 
$\tau = 50$& 2.27e-7 & 2.14e-6 & 2.70e-6  & 9.32e-4 & 1.15e-3 \tabularnewline
\hline 
$\tau = 100$& 4.45e-7 & 3.27e-6 & 3.54e-6  & 9.56e-4 & 1.15e-3 \tabularnewline
\hline 
$\tau = 200$& 7.45e-7 & 5.37e-6  & 5.48e-6 & 9.83e-4  & 1.15e-3 \tabularnewline
\hline 
$\tau = 300$& 1.02e-6 & 7.46e-6  & 7.52e-6 & 9.99e-4 & 1.15e-3 \tabularnewline
\hline 
\end{tabular}
\smallskip
\smallskip
\caption{\footnotesize (Example \ref{example1}) Errors with respect to different choices of $\tau = 1,10,50,100,200,300$ and fixed $\theta =1$ on a fixed mesh with mesh size $h=0.011$ based on ${BDM}_1$-${P}_{2}$.}\label{table2}
\end{table}

In order to test the influence of the choice of $\theta$ and $\tau$, we test this example based on the ${BDM}_1$-${P}_{2}$ mixed finite element space on a fixed mesh $\Ct_h$ with mesh size $h=0.011$. We denote $e_u = u-u_h$ and $e_{\bfw} = \bfw -\bfw_h$. Firstly, we test different choices of $\theta = -1,1$ and fixed $\tau =10$, and we always have the errors as shown in Table \ref{table1}. Actually, we also test other choices of $\theta = -0.5, 0, 0.5$, and we get the same errors as in Table \ref{table1}. Thus, for this example with smooth solution, we find that the influence of different choices of $\theta$ is small. In order to mainly test the accuracy and efficiency of the proposed algorithm, we always use $\theta = 1$ in the following tests. Next, we test the influence of different choices of $\tau = 1,10, 50, 100, 200, 300$ and $\theta=1$. From the viewpoint of theoretical analysis, the parameter $\tau$ should be chosen to be large enough. However, from Table {\ref{table2}} we can see that $\tau$ can only be mildly large to get the desired results. We let $\tau=10$ in the following tests for this example.

\begin{table}[htbp]\renewcommand{\arraystretch}{1.2}
\footnotesize
\begin{tabular}{c|c c | c c | c c| c c | c c}
\hline 
$h$ &$\|e_u\|_{\Ct_h}$ & {\rm order} & $\| \nabla e_u\|_{\Ct_h} $& order &$ \|e_{\bfw}\|_{\Ct_h} $& order  &$\|\nabla \cdot e_{\bfw}\|_{\Ct_h} $  & order  &$\|\nabla e_u\|_{1,\Ct_h} $ &  order \tabularnewline
\hline  
$0.3536$& 9.28e-4 & -- & 4.39e-3  & --  & 3.03e-3 & -- &2.82e-2  & --& 4.16e-2 & --  \tabularnewline 
$0.1768$& 2.29e-4 & 2.02 & 1.09e-3 & 2.01  & 7.89e-4  & 1.94  & 1.45e-2& 0.96 & 1.92e-2   & 1.12 \tabularnewline 
$0.0884$& 5.70e-5 & 2.01 & 2.72e-4  & 2.00  & 2.00e-4  & 1.98  & 7.27e-3& 1.00 & 9.30e-3 &1.05  \tabularnewline 
$0.0442$& 1.42e-5 &2.01 & 6.78e-5  &2.00   & 5.03e-5 &1.99   & 3.64e-3& 1.00 &4.61e-3 &  1.01 \tabularnewline 
$0.0221$& 3.56e-6 & 2.00& 1.70e-5 & 2.00  & 1.26e-5  & 2.00  & 1.82e-3&1.00  & 2.30e-3&  1.00 \tabularnewline 
$0.0110$& 8.89e-7 & 2.00& 4.24e-6  & 2.00  & 3.15e-6  & 2.00  & 9.11e-4&1.00  &1.15e-3 & 1.00 \tabularnewline 
$0.0055$& 2.22e-7 &2.00 & 1.05e-6  &2.01   & 7.88e-7 &  2.00 & 4.56e-4 &1.00 & 5.74e-4 &1.00  \tabularnewline 
\hline 
\end{tabular}
\smallskip
\smallskip
\caption{\footnotesize (Example \ref{example1}) Convergence history based on ${BDM}_1$-${P}_{2}$.}\label{table3}
\end{table}

\begin{table}[htbp]\renewcommand{\arraystretch}{1.2}
\footnotesize
\begin{tabular}{c|c c | c c | c c| c c | c c}
\hline 
$h$ &$\|e_u\|_{\Ct_h}$ & {\rm order} & $\| \nabla e_u\|_{\Ct_h} $& order &$ \|e_{\bfw}\|_{\Ct_h} $& order  &$\|\nabla \cdot e_{\bfw}\|_{\Ct_h} $  & order  &$\|\nabla e_u\|_{1,\Ct_h} $ &  order \tabularnewline
\hline  
$0.3536$& 1.08e-4 &   -- & 1.32e-3  & --  & 1.67e-3 & -- & 7.27e-3  & --&  2.80e-2 & --  \tabularnewline 
$0.1768$& 7.16e-6 & 3.91 & 1.71e-4  &2.95  & 2.15e-4 & 2.96  &1.85e-3  & 1.97 & 7.13e-3 & 1.97    \tabularnewline 
$0.0884$& 4.57e-7 & 3.97 & 2.17e-5 & 2.98  & 2.71e-5  &2.99     & 4.67e-4& 1.99  & 1.75e-3   & 2.03  \tabularnewline 
$0.0442$& 2.88e-8 & 3.99  & 2.74e-6  & 2.99   & 3.39e-6  & 3.00    & 1.17e-4& 2.00  & 4.30e-4 & 2.02  \tabularnewline 
$0.0221$& 1.83e-9 & 3.98 & 3.44e-7  & 2.99   & 4.23e-7 & 3.00   & 2.92e-5& 2.00  &1.06e-4 &2.02    \tabularnewline 
\hline 
\end{tabular}
\smallskip
\smallskip
\caption{\footnotesize (Example \ref{example1}) Convergence history based on ${BDM}_2$-${P}_{3}$.}\label{table4}
\end{table}

We further test this example based on the mixed finite element spaces ${BDM}_1$-${P}_{2}$ and ${BDM}_2$-${P}_{3}$ respectively. From Tables \ref{table3}-\ref{table4}, we can see that the errors always achieve almost optimal orders of convergence which are consistent with the theoretical analysis.

\end{example}

\begin{example}\label{example2}
In this example we test our algorithm for the biharmonic equation (\ref{equ_biharmonic}) without exact solution in two L-shape type domains $\Omega_1 = [1,2]^2\setminus [1,1.5]^2$ and $\Omega_2 = [1,2]^2\setminus ([1,4/3]\times [4/3,5/3] \cup [1,5/3]\times [1,4/3])$. Let $\kappa=1$. By an elliptic regularity result for the clamped Kirchhoff plate (cf. \cite[Lemma 1.1]{Brenner2017},\cite{BR1980}), the parameter $\delta$ given in (\ref{reg_ass}) holds that $\delta\in(1/2,1)$ for the solutions in this example. Due to the low regularity property of the solutions in this example, we only consider the approximation of the solutions based on the lowest order of mixed finite element space ${BDM}_1$-${P}_{2}$ in the proposed algorithm. 

\begin{figure}[htbp]
\begin{center}
\includegraphics[width=7cm,height=5.4cm,clip,trim=0cm 0cm 0cm 0cm]{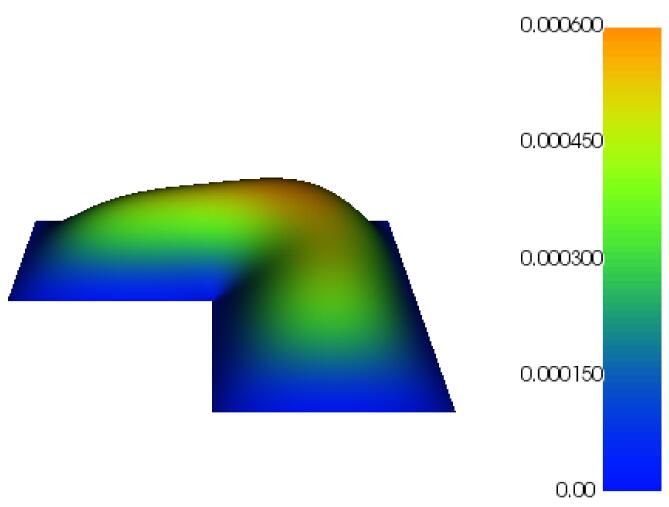}
\includegraphics[width=7cm,height=5.4cm,clip,trim=0cm 0cm 0cm 0cm]{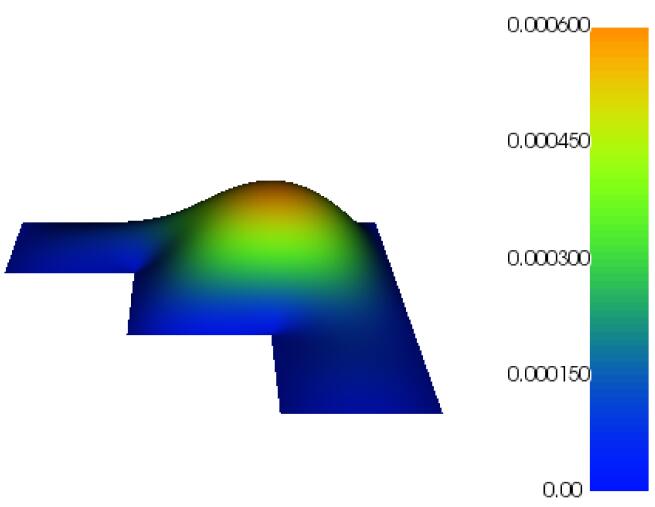}
\end{center}
\caption{\footnotesize (Example \ref{example2}) Left: The solution $u_h$ for the problem in $\Omega_1$ on the mesh with mesh size $h=0.00435$. Right: The solution $u_h$ for the problem in $\Omega_2$ on the mesh with mesh size $h=0.00435$.}\label{ex2_fig}
\end{figure}

\begin{table}[htbp]\renewcommand{\arraystretch}{1.2}
\footnotesize
\begin{tabular}{c|c c | c c | c c| c c | c c}
\hline 
$h$ &$\|e^\ast_u\|_{\Ct_h}$ & {\rm order} & $\| \nabla e^\ast_u\|_{\Ct_h} $& order &$ \|e^\ast_{\bfw}\|_{\Ct_h} $& order  &$\|\nabla \cdot e^\ast_{\bfw}\|_{\Ct_h} $  & order  &$\|\nabla e^\ast_u\|_{1,\Ct_h} $ &  order \tabularnewline
\hline  
$0.2841$& 7.96e-5 & -- & 5.68e-4  & --  & 6.38e-4 & -- & 9.83e-3  & --& 8.76e-3 & --  \tabularnewline 
$0.1368$& 2.17e-5 & 1.88  & 1.58e-4 & 1.85    & 1.75e-4  & 1.87    & 5.65e-3 & 0.80  &  5.34e-3   & 0.71  \tabularnewline 
$0.0680$& 7.95e-6 &  1.45 & 5.79e-5  & 1.45    & 6.27e-5  &  1.48  & 3.93e-3& 0.52  &  3.94e-3 & 0.44   \tabularnewline 
$0.0345$& 3.89e-6 & 1.03  & 2.74e-5  & 1.08   & 2.95e-5 & 1.09   &  2.94e-3&  0.42 & 2.80e-3 & 0.49    \tabularnewline 
\hline 
\end{tabular}
\smallskip
\smallskip
\caption{\footnotesize (Example \ref{example2}) Convergence history for the problem in $\Omega_1$.}\label{ex2_table1}
\end{table}

\begin{table}[htbp]\renewcommand{\arraystretch}{1.2}
\footnotesize
\begin{tabular}{c|c c | c c | c c| c c | c c}
\hline 
$h$ &$\|e^\ast_u\|_{\Ct_h}$ & {\rm order} & $\| \nabla e^\ast_u\|_{\Ct_h} $& order &$ \|e^\ast_{\bfw}\|_{\Ct_h} $& order  &$\|\nabla \cdot e^\ast_{\bfw}\|_{\Ct_h} $  & order  &$\|\nabla e^\ast_u\|_{1,\Ct_h} $ &  order \tabularnewline
\hline  
$0.2417$& 3.83e-5 & --      & 2.88e-4  & --  &       3.26e-4 & -- &          7.03e-3  & --& 6.48e-3 & --  \tabularnewline 
$0.1320$& 1.69e-5 & 1.18   & 1.32e-4 & 1.13     & 1.43e-4  &  1.19    & 5.01e-3 & 0.49   &  5.14e-3   & 0.33   \tabularnewline 
$0.0684$& 9.11e-6 & 0.89   & 6.44e-5  & 1.04      & 6.85e-5  &  1.06   & 3.63e-3& 0.46   &  3.53e-3 & 0.54    \tabularnewline 
$0.0347$& 3.50e-6 & 1.38   & 2.54e-5  &  1.34   & 2.71e-5 & 1.34    &  2.64e-3&  0.46  & 2.56e-3 & 0.46     \tabularnewline 
$0.0172$& 1.49e-6 &  1.23   & 1.11e-5  & 1.19    & 1.15e-5 &  1.24   &  1.70e-3&  0.64  & 1.73e-3 &  0.57    \tabularnewline 
\hline 
\end{tabular}
\smallskip
\smallskip
\caption{\footnotesize (Example \ref{example2}) Convergence history for the problem in $\Omega_2$.}\label{ex2_table2}
\end{table}

We test the case with source term $f=1$ and set the parameter $\tau=200$. Figure \ref{ex2_fig} shows the solutions of the problems in $\Omega_1$ and $\Omega_2$ respectively on the finest mesh with mesh size $h=0.00435$. Let $u^\ast$ and $\bfw^\ast$ be the approximation solutions on the finest mesh. We denote the errors $e^\ast_u = u^\ast - u_h$ and $e^\ast_\bfw = \bfw^\ast - \bfw_h$. Tables \ref{ex2_table1}-\ref{ex2_table2} indicate that the errors also achieve almost optimal orders of convergence.
\end{example}

\begin{example}\label{example3}
In this example we test our algorithm for the biharmonic equation (\ref{equ_biharmonic}) with variable coefficient which can also be assumed to satisfy the non-homogeneous boundary conditions. We consider the biharmonic equation (\ref{equ_biharmonic1}) on a unit square domain $[0,1]^2$ with the exact solution as follows:
\[
u = sin(2\pi x) sin(2\pi y).
\]
We assume $\kappa(\bfx) = x^2+y^2+1$, then the source term $f$ and the boundary conditions can be determined by the exact solution and the coefficient $\kappa(\bfx)$.

\begin{table}[htbp]\renewcommand{\arraystretch}{1.2}
\footnotesize
\begin{tabular}{c|c c | c c | c c| c c | c c}
\hline 
$h$ &$\|e_u\|_{\Ct_h}$ & {\rm order} & $\| \nabla e_u\|_{\Ct_h} $& order &$ \|e_{\bfw}\|_{\Ct_h} $& order  &$\|\nabla \cdot e_{\bfw}\|_{\Ct_h} $  & order  &$\|\nabla e_u\|_{1,\Ct_h} $ &  order \tabularnewline
\hline  
$0.3536$&  6.96e-1& -- & 6.57  & --               & 1.51 & --            &19.42  & --& 64.27 & --  \tabularnewline 
$0.1768$& 1.78e-1 &1.97   & 1.63 & 2.01   &  4.27e-1 & 1.82   & 10.19 &    0.93 & 21.18  &  1.60  \tabularnewline 
$0.0884$& 4.46e-2 &  2.00 & 4.04e-1  & 2.01   & 1.11e-1  &1.94    & 5.15  & 0.98   & 7.94 &  1.42  \tabularnewline 
$0.0442$& 1.11e-2& 2.01  & 1.01e-1 & 2.00   & 2.79e-2&   1.99 & 2.58  &  1.00  &3.52  &   1.17 \tabularnewline 
$0.0221 $& 2.78e-3 & 2.00  & 2.52e-2 & 2.00   & 7.00e-3  &  1.99  & 1.29  & 1.00  & 1.70 &1.05    \tabularnewline 
$0.0110 $& 6.95e-4 &2.00  & 6.29e-3  & 2.00   & 1.75e-3  &  2.00  & 6.46e-1 &  1.00  &8.42e-1 &  1.01 \tabularnewline 
$0.0055 $& 1.74e-4 & 2.00  & 1.57e-3 & 2.00   & 4.38e-4& 2.00   & 3.23e-1 &  1.00 & 4.20e-1 & 1.00  \tabularnewline 
\hline 
\end{tabular}
\smallskip
\smallskip
\caption{\footnotesize (Example \ref{example3}) Convergence history based on ${BDM}_1$-${P}_{2}$.}\label{ex3_table1}
\end{table}

\begin{table}[htbp]\renewcommand{\arraystretch}{1.2}
\footnotesize
\begin{tabular}{c|c c | c c | c c| c c | c c}
\hline 
$h$ &$\|e_u\|_{\Ct_h}$ & {\rm order} & $\| \nabla e_u\|_{\Ct_h} $& order &$ \|e_{\bfw}\|_{\Ct_h} $& order  &$\|\nabla \cdot e_{\bfw}\|_{\Ct_h} $  & order  &$\|\nabla e_u\|_{1,\Ct_h} $ &  order \tabularnewline
\hline  
$0.3536$& 1.94e-1 &   --          & 2.64   & --              & 2.27  & --        & 7.06   & --&  61.94  & --  \tabularnewline 
$0.1768$& 1.44e-2& 3.75     &   3.46e-1 & 2.93  &  3.21e-1 & 2.82    & 1.83 & 2.05    &14.03   & 2.14     \tabularnewline 
$0.0884$& 9.43e-4 &  3.93   & 4.38e-2 & 2.98     &  4.13e-2  &  2.96    & 4.58e-1& 2.00   & 3.34  &2.07    \tabularnewline 
$0.0442$& 5.97e-5 &  3.98    & 5.49e-3  & 3.00    & 5.21e-3  &  2.99    & 1.15e-1 &  1.99  & 8.22e-1 & 2.02   \tabularnewline 
$0.0221$& 3.74e-6 & 4.00    & 6.87e-4  &3.00     & 6.52e-4 &  3.00   &  2.86e-2& 2.01   &2.05e-1 &  2.00   \tabularnewline 
$0.0110$& 2.39e-7 & 3.97    & 8.59e-5 &  3.00   & 8.15e-5 &  3.00   &  7.16e-3 & 2.00   &5.11e-2 & 2.00    \tabularnewline 
\hline 
\end{tabular}
\smallskip
\smallskip
\caption{\footnotesize (Example \ref{example3}) Convergence history based on ${BDM}_2$-${P}_{3}$.}\label{ex3_table2}
\end{table}

We test our algorithm with $\theta=1$ and $\tau=10$ for this example based on the mixed finite element spaces ${BDM}_1$-${P}_{2}$ and ${BDM}_2$-${P}_{3}$ respectively. From Tables \ref{ex3_table1}-\ref{ex3_table2}, we can see that for this example, our algorithm always achieve almost optimal convergence.
\end{example}

Next, we start to test our algorithm for the solution of von K\'arm\'an equation (\ref{equ_von_karman}). For simplicity, in the following experiments we always test the proposed algorithm based on the mixed finite element space ${BDM}_1$-${P}_{2}$. Actually, the new algorithm can be easily extended for the solution of the von K\'arm\'an model (cf.  \cite{Brenner2017}) as follows:
\begin{subequations}\label{equ_von_karman_v1}
\begin{align}
\Delta^2 \xi - [\xi,\psi] + p \Delta \xi = f, &  \ \ \textrm{in}  \ \ \Omega, \label{vk1_eq1}\\
\Delta^2 \psi  + [\xi,\xi]= g, &  \ \ \textrm{in}  \ \ \Omega,\label{vk1_eq2} \\
\xi =\frac{\partial \xi}{\partial n}  = 0, &  \ \ \textrm{on}  \ \  \partial \Omega, \label{vk1_bc1}\\
\psi = \frac{\partial \psi}{\partial n} = 0, &  \ \ \textrm{on}  \ \  \partial \Omega,\label{vk1_bc2}
\end{align}
\end{subequations}
where $p$ is a given positive constant and the boundary conditions can also be non-homogeneous. Similar to (\ref{sta_method_vk}), we seek an approximation 
$(\bfu_h,\bfw_h,\xi_h,\psi_h)\in \bfW_h\times \bfW_h\times V_h \times V_h$ such that
\begin{subequations}\label{sta_method_vk_1}
\begin{align}
B_{\theta}((\bfu_h,\xi_h),(\bfv,\eta))+({\rm cof}(D^2\xi_h)\nabla \psi_h,\nabla \eta)_{\calT_h} - (p \nabla \xi_h, \eta)  &= (f,\eta)_{\calT_h},\\
B_{\theta} ((\bfw_h,\psi_h),(\bfz,\phi)) -({\rm cof}(D^2\xi_h)\nabla \xi_h,\nabla \phi)_{\calT_h} &= (g,\phi)_{\calT_h},
\end{align}
\end{subequations}
for any $(\bfv,\bfz,\eta,\phi) \in \bfW_h\times \bfW_h\times V_h \times V_h$. The well-posedness of (\ref{sta_method_vk_1}) and the associated error estimates can be similarly analyzed.

\begin{example}\label{example4}
In this example we test the proposed algorithm based on the formulation (\ref{sta_method_vk_1}) with $\theta =1$ for the von K\'arm\'an equation (\ref{equ_von_karman_v1}) with $p=0$ on a unit square domain $[0,1]^2$. The exact solution $(\xi,\psi)$ is assumed to be
\[
\xi = x^2(1-x)^2y^2(1-y)^2,\quad \psi = \sin^2(\pi x) \sin^2(\pi y),
\] 
and the source term $f$ in (\ref{vk1_eq1}) and the right-hand side $g$ in (\ref{vk1_eq2}) can be determined respectively. For the solution of nonlinear system, one can use nonlinear system solver such as Picard iteration, Newton iteration. In the following experiments, we always apply Picard iteration with zero initial guess to solve the nonlinear system.

\begin{table}[htbp]\renewcommand{\arraystretch}{1.2}
\footnotesize
\begin{tabular}{c|c c | c c | c c| c c | c c}
\hline 
$h$ &$\|e_\xi\|_{\Ct_h}$ & {\rm order} & $\| \nabla e_\xi\|_{\Ct_h} $& order &$ \|e_{\bfu}\|_{\Ct_h} $& order  &$\|\nabla \cdot e_{\bfu}\|_{\Ct_h} $  & order  &$\|\nabla e_\xi\|_{1,\Ct_h} $ &  order \tabularnewline
\hline  
$0.3536$&  2.06e-3 & -- &  9.55e-3  & --     &  3.69e-3 & --        & 3.43e-2  & --&  6.14e-2 & --  \tabularnewline 
$0.1768$& 4.26e-4 &2.27  & 1.95e-3& 2.29  & 8.17e-4 & 2.18  & 1.50e-2&1.19  & 2.10e-2  &1.55   \tabularnewline 
$0.0884$& 1.01e-4 & 2.08  & 4.64e-4  &  2.07 & 1.99e-4  & 2.04  & 7.33e-3&1.03  & 9.51e-3  &1.14    \tabularnewline 
$0.0442$& 2.50e-5 & 2.01 &  1.15e-4 & 2.01  & 4.95e-5 & 2.01   & 3.65e-3 & 1.01  &  4.63e-3 & 1.04    \tabularnewline 
$0.0221$& 6.24e-6 & 2.00   & 2.86e-5& 2.01    & 1.24e-5 &2.00   & 1.82e-3 & 1.00   &  2.30e-3 &  1.01   \tabularnewline 
$0.0110$& 1.56e-6 & 2.00  & 7.14e-6  &  2.00  & 3.09e-6 &  2.00  & 9.11e-4 &1.00    &  1.15e-3 &1.00    \tabularnewline 
\hline 
\end{tabular}
\smallskip
\smallskip
\caption{\footnotesize (Example \ref{example4}) Convergence history for the approximation of $\xi$ and $\bfu$.}\label{ex4_table1}
\end{table}

\begin{table}[htbp]\renewcommand{\arraystretch}{1.2}
\footnotesize
\begin{tabular}{c|c c | c c | c c| c c | c c}
\hline 
$h$ &$\|e_\psi\|_{\Ct_h}$ & {\rm order} & $\| \nabla e_\psi\|_{\Ct_h} $& order &$ \|e_{\bfw}\|_{\Ct_h} $& order  &$\|\nabla \cdot e_{\bfw}\|_{\Ct_h} $  & order  &$\|\nabla e_\psi\|_{1,\Ct_h} $ &  order \tabularnewline
\hline  
$0.3536$& 1.68e-1   &   -- &  1.00  & --         & 5.70e-1  & -- & 6.01  & --& 11.93 & --  \tabularnewline 
$0.1768$& 4.20e-2 & 2.00  & 2.45e-1 &2.03    &1.55e-1&1.88    & 3.12 &  0.95  & 4.70 & 1.34   \tabularnewline 
$0.0884$& 1.05e-2& 2.00 & 6.06e-2 &  2.02   & 3.97e-2&  1.97  & 1.58 & 0.98  & 2.11 &   1.16   \tabularnewline 
$0.0442$& 2.62e-3 & 2.00 & 1.51e-2  &   2.00   & 9.99e-3&1.99     & 7.92e-1 &1.00    &  1.02 & 1.05    \tabularnewline 
$0.0221$&  6.54e-4 &  2.00 & 3.78e-3 &2.00      & 2.50e-3&  2.00   & 3.96e-1 & 1.00    & 5.06e-1 &1.01      \tabularnewline 
$0.0110$& 1.64e-4 & 2.00  & 9.44e-4 & 2.00   &  6.26e-4  & 2.00   & 1.98e-1& 1.00   &2.52e-1&   1.01 \tabularnewline 
\hline 
\end{tabular}
\smallskip
\smallskip
\caption{\footnotesize (Example \ref{example4}) Convergence history for the approximation of $\psi$ and $\bfw$.}\label{ex4_table2}
\end{table}

We denote $e_\xi = \xi-\xi_h$, $e_{\bfu} = \bfu -\bfu_h$, $e_\psi = \psi-\psi_h$, $e_{\bfw} = \bfw -\bfw_h$. We test our algorithm with $\tau=10$ for this example. In fact, the exact solutions of this example are sufficiently smooth, and Tables \ref{ex4_table1}-\ref{ex4_table2} further show that the errors always achieve almost optimal orders of convergence as the theoretical results.
\end{example}

\begin{example}\label{example5}
Now we test our algorithm for the von K\'arm\'an model (\ref{equ_von_karman_v1}) in the L-shape type domain $\Omega = [1,2]^2\setminus ([1.5,2]\times[1,1.5])$. We assume $f=10$ and $g=10$ in (\ref{equ_von_karman_v1}) for this example and choose the parameter $\tau=200$ in the proposed algorithm for this example.

\begin{table}[htbp]\renewcommand{\arraystretch}{1.2}
\footnotesize
\begin{tabular}{c|c c | c c | c c| c c | c c}
\hline 
$h$ &$\|e^\ast_\xi\|_{\Ct_h}$ & {\rm order} & $\| \nabla e^\ast_\xi\|_{\Ct_h} $& order &$ \|e^\ast_{\bfu}\|_{\Ct_h} $& order  &$\|\nabla \cdot e^\ast_{\bfu}\|_{\Ct_h} $  & order  &$\|\nabla e^\ast_\xi\|_{1,\Ct_h} $ &  order \tabularnewline
\hline  
$0.2588$& 7.09e-4 & -- &   5.08e-3  & --   &   5.55e-3   & --        &  8.92e-2  & --&   8.18e-2 & --  \tabularnewline 
$0.1329$& 2.49e-4  & 1.51   & 1.80e-3 & 1.50   & 1.98e-3  & 1.49   & 5.86e-2 &  0.61 &  5.44e-2  &   0.59 \tabularnewline 
$0.0656$& 9.86e-5  & 1.34   & 6.90e-4   & 1.38   &  7.42e-4 & 1.42   &  3.88e-2 & 0.59  & 3.72e-2   &  0.55   \tabularnewline 
$0.0329$& 4.97e-5  & 0.99   &  3.39e-4  & 1.03   &  3.54e-4 &  1.07   &   2.55e-2 &  0.61  &   2.44e-2  & 0.61     \tabularnewline 
$0.0166$&  1.33e-5 & 1.90    & 9.73e-5 & 1.80      &  1.00e-4 & 1.82   & 1.40e-2   &  0.87   &  1.47e-2  & 0.73      \tabularnewline 
\hline 
\end{tabular}
\smallskip
\smallskip
\caption{\footnotesize (Example \ref{example5}) Convergence history for the approximation of $\xi$ and $\bfu$ for the case with $p=0$.}\label{ex5_table1}
\end{table}

\begin{table}[htbp]\renewcommand{\arraystretch}{1.2}
\footnotesize
\begin{tabular}{c|c c | c c | c c| c c | c c}
\hline 
$h$ &$\|e^\ast_\psi\|_{\Ct_h}$ & {\rm order} & $\| \nabla e^\ast_\psi\|_{\Ct_h} $& order &$ \|e^\ast_{\bfw}\|_{\Ct_h} $& order  &$\|\nabla \cdot e^\ast_{\bfw}\|_{\Ct_h} $  & order  &$\|\nabla e^\ast_\psi\|_{1,\Ct_h} $ &  order \tabularnewline
\hline  
$0.2588$&   7.09e-4 & -- &  5.08e-3  & --     &  5.55e-3 & --        & 8.91e-2  & --&  8.18e-2 & --  \tabularnewline 
$0.1329$& 2.49e-4  & 1.51   &  1.80e-3 & 1.50   &  1.98e-3 & 1.49   & 5.86e-2 & 0.60  &  5.43e-2 &  0.59  \tabularnewline 
$0.0656$& 9.86e-5 &  1.34  & 6.90e-4  &1.38    &   7.42e-4 &   1.42 &  3.88e-2 & 0.59  &  3.72e-2  &  0.55   \tabularnewline 
$0.0329$&  4.97e-5 &  0.99 &  3.39e-4  & 1.03   & 3.54e-4 &  1.07   &  2.55e-2 &  0.61  &  2.44e-2  &  0.61    \tabularnewline 
$0.0166$&   1.33e-5 &  1.90   & 9.72e-5  &1.80      & 1.00e-4  & 1.82    & 1.40e-2  &  0.87   &  1.47e-2  &   0.73   \tabularnewline 
\hline 
\end{tabular}
\smallskip
\smallskip
\caption{\footnotesize (Example \ref{example5}) Convergence history for the approximation of $\psi$ and $\bfw$ for the case with $p=0$.}\label{ex5_table2}
\end{table}

\begin{table}[htbp]\renewcommand{\arraystretch}{1.2}
\footnotesize
\begin{tabular}{c|c c | c c | c c| c c | c c}
\hline 
$h$ &$\|e^\ast_\xi\|_{\Ct_h}$ & {\rm order} & $\| \nabla e^\ast_\xi\|_{\Ct_h} $& order &$ \|e^\ast_{\bfu}\|_{\Ct_h} $& order  &$\|\nabla \cdot e^\ast_{\bfu}\|_{\Ct_h} $  & order  &$\|\nabla e^\ast_\xi\|_{1,\Ct_h} $ &  order \tabularnewline
\hline  
$0.2588$& 1.34e-3 & -- &   9.25e-3  & --   &   9.81e-3   & --        &  1.34e-1  & --&   1.22e-1 & --  \tabularnewline 
$0.1329$& 4.98e-4 &  1.43   & 3.40e-3 & 1.44    & 3.62e-3 & 1.44    & 8.46e-2 &  0.66  &  7.83e-2  &  0.64   \tabularnewline 
$0.0656$& 1.94e-4 &  1.36   & 1.29e-3  &  1.40   &  1.35e-3 &  1.42   &  5.46e-2 &   0.63 &  5.24e-2   &  0.58    \tabularnewline 
$0.0329$& 9.49e-5  &  1.03   &  6.16e-4  & 1.07     &  6.34e-4 & 1.09     &   3.57e-2 & 0.61    &   3.41e-2  & 0.62      \tabularnewline 
$0.0166$&  2.56e-5 &  1.89    & 1.75e-4 &   1.82     &  1.79e-4 & 1.82    & 1.96e-2  &  0.87    &  2.05e-2  &   0.73     \tabularnewline 
\hline 
\end{tabular}
\smallskip
\smallskip
\caption{\footnotesize (Example \ref{example5}) Convergence history for the approximation of $\xi$ and $\bfu$ for the case with $p=20$.}\label{ex5_table3}
\end{table}

\begin{table}[htbp]\renewcommand{\arraystretch}{1.2}
\footnotesize
\begin{tabular}{c|c c | c c | c c| c c | c c}
\hline 
$h$ &$\|e^\ast_\psi\|_{\Ct_h}$ & {\rm order} & $\| \nabla e^\ast_\psi\|_{\Ct_h} $& order &$ \|e^\ast_{\bfw}\|_{\Ct_h} $& order  &$\|\nabla \cdot e^\ast_{\bfw}\|_{\Ct_h} $  & order  &$\|\nabla e^\ast_\psi\|_{1,\Ct_h} $ &  order \tabularnewline
\hline  
$0.2588$&   7.08e-4 & -- &  5.08e-3  & --     &  5.54e-3 & --        & 8.91e-2  & --&   8.18e-2 & --  \tabularnewline 
$0.1329$& 2.49e-4  &  1.51   &  1.80e-3 &1.50     &  1.98e-3 & 1.48    & 5.86e-2 & 0.60    &  5.43e-2 & 0.59    \tabularnewline 
$0.0656$& 9.86e-5 &   1.34  &6.90e-4 & 1.38    &   7.42e-4 &   1.42  &  3.88e-2 &  0.59  &  3.72e-2  &   0.55   \tabularnewline 
$0.0329$&  4.96e-5 &  0.99  & 3.39e-4  & 1.03    & 3.54e-4  & 1.07     & 2.55e-2 &  0.61   &  2.44e-2  &   0.61    \tabularnewline 
$0.0166$&   1.33e-5 &   1.90   & 9.72e-5  &    1.80   & 1.00e-4  & 1.82     & 1.40e-2  &  0.87    &  1.47e-2  &  0.73     \tabularnewline 
\hline 
\end{tabular}
\smallskip
\smallskip
\caption{\footnotesize (Example \ref{example5}) Convergence history for the approximation of $\psi$ and $\bfw$ for the case with $p=20$.}\label{ex5_table4}
\end{table}

We test two cases of von K\'arm\'an model (\ref{equ_von_karman_v1}) with $p=0$ and $p=20$ respectively. Let $\xi^\ast,\bfu^\ast,\psi^\ast,\bfw^\ast$ be the approximation solutions on the finest mesh with mesh size $h=0.00831$. We denote the errors $e_\xi^\ast = \xi^\ast - \xi^\ast_h$, $e^\ast_{\bfu} = {\bfu}^\ast - {\bfu}_h$, $e_{\psi}^\ast =  \psi^\ast - \psi^\ast_h$, $e^\ast_\bfw = \bfw^\ast - \bfw_h$.

Tables \ref{ex5_table1}-\ref{ex5_table2} show the errors for the case with $p=0$, and the errors for the case with $p=20$ are shown in Tables \ref{ex5_table3}-\ref{ex5_table4}. As the regularity result in Example \ref{example2}, the parameter $\beta$ in (\ref{ass_reg_vk}) holds for $\beta \in (1/2,1)$ for the solutions in this example. We can see from Tables \ref{ex5_table1}-\ref{ex5_table4} that the errors from different cases achieve almost optimal orders of convergence.

\end{example}

\begin{example}\label{example6}
We further test our algorithm for the von K\'arm\'an model (\ref{equ_von_karman_v1}) with $p=20$ in another L-shape type domain $\Omega =  [1,2]^2\setminus ([1,4/3]\times [4/3,5/3] \cup [1,5/3]\times [1,4/3])$. We assume $f=100$ and $g=1$ in (\ref{equ_von_karman_v1}) for this example and also choose the parameter $\tau=200$ in the proposed algorithm for this example.

\begin{table}[htbp]\renewcommand{\arraystretch}{1.2}
\footnotesize
\begin{tabular}{c|c c | c c | c c| c c | c c}
\hline 
$h$ &$\|e^\ast_\xi\|_{\Ct_h}$ & {\rm order} & $\| \nabla e^\ast_\xi\|_{\Ct_h} $& order &$ \|e^\ast_{\bfu}\|_{\Ct_h} $& order  &$\|\nabla \cdot e^\ast_{\bfu}\|_{\Ct_h} $  & order  &$\|\nabla e^\ast_\xi\|_{1,\Ct_h} $ &  order \tabularnewline
\hline  
$0.2313$&  7.95e-3   & -- &  5.54e-2     & --   & 5.93e-2 & --        &  9.83e-1     & --& 8.99e-1     & --  \tabularnewline 
$0.1270$&  3.64e-3 &   1.13   &  2.58e-2 &1.10      & 2.70e-2  &   1.14   & 6.71e-1  & 0.55     &   6.84e-1 & 0.39     \tabularnewline 
$0.0658 $& 1.87e-3  & 0.96     & 1.25e-2   & 1.05     &  1.29e-2   &   1.07   & 4.60e-1   &  0.54   &   4.49e-1   &   0.61    \tabularnewline 
$0.0335 $& 7.33e-4   &  1.35    & 4.98e-3    &  1.33     & 5.12e-3   &    1.33   &   3.13e-1  & 0.55     &  3.08e-1    &   0.54     \tabularnewline 
$0.0166 $&  2.23e-4  &  1.72     &  1.56e-3 &   1.67      &  1.59e-3  &  1.69    &  1.64e-1  &   0.93    &  1.71e-1   &   0.85      \tabularnewline 
\hline 
\end{tabular}
\smallskip
\smallskip
\caption{\footnotesize (Example \ref{example6}) Convergence history for the approximation of $\xi$ and $\bfu$.}\label{ex6_table1}
\end{table}

\begin{table}[htbp]\renewcommand{\arraystretch}{1.2}
\footnotesize
\begin{tabular}{c|c c | c c | c c| c c | c c}
\hline 
$h$ &$\|e^\ast_\psi\|_{\Ct_h}$ & {\rm order} & $\| \nabla e^\ast_\psi\|_{\Ct_h} $& order &$ \|e^\ast_{\bfw}\|_{\Ct_h} $& order  &$\|\nabla \cdot e^\ast_{\bfw}\|_{\Ct_h} $  & order  &$\|\nabla e^\ast_\psi\|_{1,\Ct_h} $ &  order \tabularnewline
\hline  
$0.2313$&  1.89e-4 & -- &  1.43e-3    & --   & 1.45e-3      & --        &    1.46e-2  & --&   1.35e-2  & --  \tabularnewline 
$0.1270$&  9.70e-5 & 0.96     & 7.41e-4  &   0.95   &7.50e-4   &  0.95    &   8.34e-3&  0.81   & 8.05e-3   &  0.75    \tabularnewline 
$0.0658 $&  4.94e-5 &  0.97    & 3.71e-4   & 1.00      &3.74e-4&  1.00    & 4.45e-3   &  0.91   &  4.48e-3    &  0.85     \tabularnewline 
$0.0335 $& 2.07e-5   & 1.25     & 1.52e-4    & 1.29      &  1.53e-4  & 1.29      &   2.10e-3  & 1.08     & 2.31e-3     & 0.96        \tabularnewline 
$0.0166 $&  6.77e-6  &  1.61     &4.88e-5   &  1.64       &4.89e-5    & 1.65      &   8.64e-4 &  1.28     & 1.06e-3   &   1.12      \tabularnewline 
\hline 
\end{tabular}
\smallskip
\smallskip
\caption{\footnotesize (Example \ref{example6}) Convergence history for the approximation of $\psi$ and $\bfw$.}\label{ex6_table2}
\end{table}

Since there are not exact solutions for this example, we also use the approximation solutions on the finest mesh with mesh size $h=0.00843$ to test the convergence of the proposed algorithm. For the regularity result of the solutions in this example, the parameter $\beta$ in (\ref{ass_reg_vk}) holds for $\beta \in (1/2,1)$. We can see from Tables \ref{ex6_table1}-\ref{ex6_table2} that most of the convergence rates of different kinds of errors are nearly optimal.

\end{example}

\section{Conclusions}\label{sec:conclusion}
We propose a new mixed finite element scheme using element-wise stabilization for the biharmonic equation on 
Lipshcitz polyhedral domains in any dimension. When solving the biharmonic equation, one merit of this scheme is 
that it produces symmetric and positive definite linear system, and the discrete $H^{2}$-stability and 
optimal convergence are obtained. Moreover, we extend the new method to solve the von K\'arm\'an equations. 
The existence, uniqueness and stability for the nonlinear system based on the new scheme, and the $H^{2}$-optimal 
convergence rate are also obtained. For the numerical experiments that we have performed thus far, our new scheme has  
desired  efficiency and convergence rates, when solving the biharmonic equation and the von K\'arm\'an equations.

\providecommand{\href}[2]{#2}

\end{document}